\documentclass[12pt]{amsart}
\usepackage{a4wide}
\usepackage[latin9]{inputenc}
\usepackage{amsmath}
\usepackage{amsfonts}
\usepackage{amssymb}
\usepackage{amsthm}
\usepackage{subfigure}
\newtheorem{theo}{Theorem}[section]
\newtheorem{prop}[theo]{Proposition}
\newtheorem{coro}[theo]{Corollary}
\newtheorem{lemm}[theo]{Lemma}

\theoremstyle{definition}

\theoremstyle{remark}
\newtheorem{rema}[theo]{Remark}
\newcommand{\Op}{\operatorname{Op}}

\newcommand{\nwc}{\newcommand}
\nwc{\eps}{\epsilon}
\nwc{\ep}{\epsilon}
\nwc{\vareps}{\varepsilon}
\nwc{\Oph}{\operatorname{Op}_\hbar}
\nwc{\la}{\langle}
\nwc{\ra}{\rangle}

\nwc{\mf}{\mathbf} 
\nwc{\blds}{\boldsymbol} 
\nwc{\ml}{\mathcal} 

\nwc{\defeq}{\stackrel{\rm{def}}{=}}

\nwc{\cE}{\ml{E}}
\nwc{\cN}{\ml{N}}
\nwc{\cO}{\ml{O}}
\nwc{\cP}{\ml{P}}
\nwc{\cU}{\ml{U}}
\nwc{\cV}{\ml{V}}
\nwc{\cW}{\ml{W}}
\nwc{\tU}{\widetilde{U}}
\nwc{\IN}{\mathbb{N}}
\nwc{\IR}{\mathbb{R}}
\nwc{\IZ}{\mathbb{Z}}
\nwc{\IC}{\mathbb{C}}
\nwc{\IT}{\mathbb{T}}
\nwc{\tP}{\widetilde{P}}
\nwc{\tPi}{\widetilde{\Pi}}
\nwc{\tV}{\widetilde{V}}
\nwc{\supp}{\operatorname{supp}}
\nwc{\rest}{\restriction}

\begin{document}

\title[Perturbed Schr\"odinger equation on negatively curved surfaces]{Long-time dynamics of the perturbed Schr\"odinger equation on negatively curved surfaces}

\author[Gabriel Rivi\`ere]{Gabriel Rivi\`ere}

\address{Laboratoire Paul Painlev\'e (U.M.R. CNRS 8524), U.F.R. de Math\'ematiques, Universit\'e Lille 1, 59655 Villeneuve d'Ascq Cedex, France}
\email{gabriel.riviere@math.univ-lille1.fr}

\begin{abstract} 
We consider perturbations of the semiclassical Schr\"odinger equation on a compact Riemannian surface with constant negative curvature and without boundary. We show that, for scales of times which are logarithmic in the size of the perturbation, the solutions associated to initial data in a small spectral window become equidistributed in the semiclassical limit. As an application of our method, we also derive some properties of the quantum Loschmidt echo below and beyond the Ehrenfest time for initial 
data in a small spectral window.
\end{abstract}

\maketitle

\section{Introduction}

In this article, we consider $(M,g)$ a smooth ($\ml{C}^{\infty}$), connected, orientable, boundaryless, compact and Riemannian manifold 
of dimension $d$. We want to study the long time dynamics of the following family of Schr\"odinger equations:
\begin{equation}\label{e:schrodinger}
\forall 0<\hbar\leq 1,\qquad 
\imath\hbar\frac{\partial u_{\hbar}}{\partial t}=\hat{P}(\hbar) u_{\hbar},\ \text{with}\ u_{\hbar}\rceil_{t=0}=\psi_{\hbar},
\end{equation}
where $(\psi_{\hbar})_{0<\hbar\leq 1}$ is a sequence of \emph{normalized} initial data in $L^2(M)$ satisfying proper oscillatory assumptions, and
\begin{equation}\label{e:schr-op}
 \hat{P}(\hbar):=-\frac{\hbar^2\Delta_g}{2}+\eps_{\hbar} V,
\end{equation}
with $\Delta_g$ the Laplace Beltrami operator induced by the Riemannian metric $g$, $V\in\ml{C}^{\infty}(M,\IR)$ and $\eps_{\hbar}\rightarrow 0$ as $\hbar\rightarrow 0^+$. 
We aim at describing the dynamics of these equations in the semiclassical limit $\hbar\rightarrow 0^+$. Due to the semiclassical approximation, these properties will be related to 
the properties of the geodesic flow $(G_0^t)_{t\in\IR}$ acting on $T^*M$. For instance, if we denote by $u_{\hbar}(\tau)$ the solution of~\eqref{e:schrodinger} at time $\tau$, 
i.e.
$$u_{\hbar}(\tau):=e^{-\frac{i\tau\hat{P}(\hbar)}{\hbar}}\psi_{\hbar},$$
then we can introduce its ``Wigner distribution''\footnote{This terminology is often reserved to the case where we consider $\IR^d$.} on $T^*M$ 
\begin{equation}\label{e:wigner}\forall a\in \ml{C}^{\infty}_c(T^*M),\ \mu_{\hbar}(\tau)(a):=\left\la u_{\hbar}(\tau),
\Oph(a)u_{\hbar}(\tau)\right\ra,
\end{equation}
where $\Oph(a)$ is a $\hbar$-pseudodifferential operator with principal symbol $a$ -- see~\cite{Zw12}. This quantity describes the distribution of 
the solution of~\eqref{e:schrodinger} in $T^*M$. As an application of the Egorov theorem~\cite{Zw12}, one finds that, for every fixed $\tau$ in $\IR$,
\begin{equation}\label{e:egorov}
\forall a\in \ml{C}^{\infty}_c(T^*M),\ \mu_{\hbar}(\tau)(a)=\mu_{\hbar}(0)(a\circ G_0^{\tau})+o(1),\ \text{as}\ \hbar\rightarrow 0^+.
\end{equation}
In other words, the ``Wigner distribution'' at time $\tau$ is related to the one at time $0$ through the action of the geodesic flow. This simple relation illustrates the 
connection between the quantum evolution and the classical one in the semiclassical limit $\hbar\rightarrow 0^+$. 


Our goal is to study the long time dynamics of~\eqref{e:schrodinger} through these distributions. For that purpose, we will study the properties of $\mu_{\hbar}(\tau)$ 
where $\tau=\tau_{\hbar}$ will depend on $\hbar>0$, and more specifically when it will tend to $+\infty$ as $\hbar\rightarrow 0^+$. We will focus on geometric 
situations where the geodesic flow enjoys some chaotic features, e.g. the Anosov property~\cite{KaHa, Rug07} -- the main examples being negatively curved 
manifolds~\cite{Ano67}. These considerations are related to questions arising in the field of quantum chaos, where one wants to understand the influence of the 
chaotic properties of the classical system on its quantum counterpart.

In the definition of the Schr\"odinger operator~\eqref{e:schr-op}, we added a self-adjoint perturbation $\eps_{\hbar}V$, where $\eps_{\hbar}\rightarrow 0$. 
In fact, we will not look at the dynamics of the ``free'' Schr\"odinger equation (where $V\equiv 0$), and our goal is rather to understand the influence of 
this kind of self-adjoint perturbations on the long time dynamics. Looking at the influence of selfadjoint perturbations on the quantum dynamics of chaotic 
systems is related to the question of the quantum Loschmidt echo in the physics literature -- see section~\ref{s:loschmidt} for a brief reminder 
or~\cite{GPSZ06, JaPe09, GJPW12} for recent surveys on these issues. In the mathematics literature, these kind of considerations have recently appeared 
in several places. In~\cite{BolSc06, CoRo07}, the authors were interested by questions directly related to the quantum Loschmidt echo. 
In~\cite{EsTo12, CanJaTo12}, the authors looked at magnetic (or metric) perturbations of the Schr\"odinger operator on a general compact manifold, 
and they obtained some informations on the pointwise bounds of the solutions of the perturbed Schr\"odinger equation for a \emph{finite time}, 
and for a typical choice of perturbations. In~\cite{EsRi14}, the questions were close to the ones considered in the present article, and we will compare 
more precisely below our results to those from this reference. The tools used here are in fact the continuation of the ones introduced in~\cite{EsRi14}. As another application of the methods from the present article, we will also deduce some properties on the quantum Loschmidt echo below and beyond the Ehrenfest time -- see section~\ref{s:loschmidt}. Finally, the long time dynamics of the ``perturbed'' Schr\"odinger equation for integrable systems is studied in~\cite{MacRiv14}.

\begin{rema} We underline that, even if the perturbations we will consider will be small in the semiclassical limit, they will be quite strong at the quantum level as we will require in our statements that $\eps_{\hbar}\geq\hbar^{\nu}$ for some $0<\nu<1/2$. In fact, according to the semiclassical Weyl law~\cite{Zw12}, the ``mean level spacing'' for the eigenvalues is of order $\hbar^{d}$ where $d$ is the dimension of $M$. 
\end{rema}

\section{Statement of the main results}

When the geodesic flow satisfies some chaotic properties, one of the classical results on the semiclassical distribution of the solutions of the ``unperturbed'' Schr\"odinger equation is the quantum ergodicity theorem~\cite{Sh74, Ze87, CdV85, HeMaRo87, Zw12}. This property is usually formulated for stationary solutions of the Schr\"odinger equations but it can be generalized to more general solutions~\cite{AnRi12}. In order to state this result, we fix some sequence $(\delta_{\hbar})_{0<\hbar\leq 1}$ which satisfies $\delta_{\hbar}\rightarrow 0$ as $\hbar\rightarrow 0^+$ and $\delta_{\hbar}\geq \alpha\hbar$ for some fixed constant $\alpha>0$. We define then
$$\ml{H}_{\hbar}:=\mathbf{1}_{[1-\delta_{\hbar},1+\delta_{\hbar}]}\left(-\hbar^2\Delta_g\right)L^2(M).$$
If the Liouville measure $L$ is ergodic for the geodesic flow on the unit cotangent bundle\footnote{Here $p_0(x,\xi):=\frac{g^*_{x}(\xi,\xi)}{2}$ with $g^*$ the metric induced by $g$ on the cotangent bundle.}
$$S^*M:=\left\{(x,\xi)\in T^*M: p_0(x,\xi)=\frac{1}{2}\right\},$$
then it is known that this space is of dimension $N(\hbar)\sim C_M\hbar^d/\delta_{\hbar}$, for some fixed constant $C_M>0$ depending only on $(M,g)$~\cite{DiSj99}. The quantum ergodicity theorem can be stated as follows~\cite{Sh74, Ze87, CdV85, HeMaRo87, Zw12, AnRi12}:
\begin{theo}[Quantum Ergodicity] Let $(\tau_{\hbar})_{0<\hbar\leq 1}$ such that $lim_{\hbar\rightarrow 0^+}\tau_{\hbar}=+\infty$. Suppose that the Liouville measure $L$ is ergodic for the geodesic flow on $S^*M$. Then, for every orthonormal basis $(\psi_{\hbar}^j)_{j=1,\ldots, N(\hbar)}$ of $\ml{H}_{\hbar}$, one can find $J(\hbar)\subset\{1,\ldots, N(\hbar)\}$ such that
$$\lim_{\hbar\rightarrow 0^+}\frac{\sharp J(\hbar)}{N(\hbar)}=1,$$ 
and, for every $a$ in $\ml{C}^{\infty}_c(T^*M)$, for every $\varphi$ in $L^1(\IR)$, one has
$$\lim_{\hbar\rightarrow 0,j\in J(\hbar)}\int_{\IR}\varphi(t)\left\la \psi_{\hbar}^j,
e^{-it\tau_{\hbar}\hbar\Delta_g}\Oph(a)e^{it\tau_{\hbar}\hbar\Delta_g}\psi_{\hbar}^j\right\ra dt=\int_{\IR}\varphi(t)dt\times\int_{S^*M}adL.$$

\end{theo}

This theorem tells us that, under ergodicity of the Liouville measure, the solutions of the ``free'' Schr\"odinger equation become equidistributed for a ``generic choice'' of initial data microlocalized near $S^*M$. We also observe that equidistribution occurs when we average over the time parameter $t$. This result is very robust and it holds for many quantum systems with an underlying chaotic classical system.

Again, this result holds for a typical choice of initial data, and it is a difficult problem to understand what can be said for \emph{a fixed sequence} of normalized initial data $(\psi_{\hbar})_{0<\hbar\leq 1}$ -- see \cite{Ze10, Sa11, No13} for recent reviews on these questions. We mention that, for short scales of times (which are at most logarithmic in $\hbar>0$), one can also get a good description of the distributions $(\mu_{\hbar}(\tau_{\hbar}))_{0<\hbar\leq 1}$ for certain class of initial data, namely coherent states~\cite{CoRo97, BoDB00, BouDB05}, or Lagrangian states~\cite{Sch05, AnNo07, An08}.

Instead of looking at particular families, or at generic families of initial data for the unperturbed equation, we will now try to understand, \emph{for general sequences of initial data}, the quantum evolution under \emph{the perturbed Schr\"odinger} equation~\eqref{e:schrodinger}, i.e. when $V\neq 0$. This question was already discussed in~\cite{EsRi14}. 
Precisely, we will now consider sequences of initial data $(\psi_{\hbar})_{0<\hbar\leq 1}$ which satisfy the following property:
\begin{equation}\label{e:hosc}
\lim_{R\rightarrow+\infty}\limsup_{\hbar\rightarrow 0^+}\left\Vert \mathbf{1}_{\left[  1-R\hbar,1+R\hbar\right]
}\left(  -\hbar^{2}\Delta\right)  \psi_{\hbar}-\psi_{\hbar}\right\Vert _{L^{2}\left(  M\right)
}\longrightarrow 0,\  \text{and}\ \forall 0<\hbar\leq 1,\ \|\psi_{\hbar}\|_{L^2}=1.
\end{equation}
\begin{rema}
Recall that in~\cite{EsRi14}, it was only required that, for every $\delta_0>0$,
\begin{equation}\label{e:hosc0}
\lim_{\hbar\rightarrow 0^+}\left\Vert \mathbf{1}_{\left[  1-\delta_0,1+\delta_0\right]
}\left(  -\hbar^{2}\Delta_g\right)  \psi_{\hbar}-\psi_{\hbar}\right\Vert _{L^{2}\left(  M\right)
}=0,\ \text{and}\ \forall 0<\hbar\leq 1,\ \|\psi_{\hbar}\|_{L^2}=1.
\end{equation}
which is more general than~\eqref{e:hosc}. In section~\ref{s:semiclassical}, we will in fact consider slightly more general initial data than the ones satisfying~\eqref{e:hosc}; yet, the assumptions will still be more restricitive than~\eqref{e:hosc0}.
\end{rema}
Our goal is to understand the action for short logarithmic times of the ``perturbed'' Schr\"odinger propagator $e^{-\frac{i\tau_{\hbar}\hat{P}(\hbar)}{\hbar}}$ on initial data satisfying~\eqref{e:hosc}, and to show that imposing~\eqref{e:hosc} instead of~\eqref{e:hosc0} allows to improve substantially the results from~\cite{EsRi14} -- see corollary~\ref{c:coro1} below. In section~\ref{s:loschmidt}, we will also show the relevance of this approach in the study of the quantum Loschmidt echo.


\subsection{Semiclassical measures} Let $(\tau_{1}(\hbar))_{0<\hbar\leq 1}$ and $(\tau_{2}(\hbar))_{0<\hbar\leq 1}$ be two sequences which satisfy $\tau_1(\hbar)\leq \tau_2(\hbar)$. We define then
$$\ml{M}([\tau_1(\hbar),\tau_2(\hbar)],\hbar\rightarrow 0^+),$$
as the set of accumulation points in $\ml{D}'(T^*M)$ (as $\hbar\rightarrow 0^+$) of the sequences of distributions $(\mu_{\hbar}(\tau_{\hbar}))_{0<\hbar\leq 1}$ defined by~\eqref{e:wigner} where 
\begin{itemize}
\item $(\psi_{\hbar})_{0<\hbar\leq 1}$ varies among sequences satisfying~\eqref{e:hosc},
\item $(\tau_{\hbar})_{0<\hbar\leq 1}$ varies among sequences satisfying $\tau_1(\hbar)\leq\tau_{\hbar}\leq\tau_2(\hbar)$ for $\hbar>0$ small enough.
\end{itemize}
Any element in $\ml{M}([\tau_1(\hbar),\tau_2(\hbar)],\hbar\rightarrow 0^+)$ is in fact a probability measure carried on $S^*M$ that is called a semiclassical (defect) measure~\cite{Ge91, Burq97, Zw12}. We underline that these measures have a priori no extra properties like invariance by the geodesic flow. Our goal is to describe the properties of the measures belonging to $\ml{M}([\tau_1(\hbar),\tau_2(\hbar)],\hbar\rightarrow 0^+)$. Given a subsequence $(\psi_{\hbar_n})_{n\in\mathbb{N}}$ satisfying~\eqref{e:hosc} with $\hbar_n\rightarrow 0^+$, we define the following subset of $\ml{M}([\tau_1(\hbar),\tau_2(\hbar)],\hbar\rightarrow 0^+)$:
$$\ml{M}(\psi_{\hbar_n},[\tau_1(\hbar_n),\tau_2(\hbar_n)],\hbar_n\rightarrow 0^+),$$
where we restricted ourselves to the sequence of initial data $(\psi_{\hbar_n})_{n\in\mathbb{N}}$.

\begin{rema}
For a given sequence $(\psi_{\hbar})_{0<\hbar\leq 1}$ satisfying~\eqref{e:hosc}, one can also extract a subsequence $\hbar_n\rightarrow 0^+$ such that the subsequence $(\mu_{\hbar_n}(0))_{n\in\mathbb{N}}$ converge in $\ml{D}'(T^*M)$ to some accumulation point $\mu_0$ which is a probability measure on $S^*M$. We note that, if we make the \emph{stronger assumption} that
 $$\left\|(-\hbar^2\Delta_g-1)\psi_{\hbar}\right\|_{L^2}=o\left(\hbar\right),\ \|\psi_{\hbar}\|_{L^2}=1,$$
then the accumulation point $\mu_0$ is in fact invariant by the geodesic flow $G_0^t$ acting on $S^*M$~\cite{Zw12} -- Ch.~$5$. In the case of a compact congruence 
surface, it was proved that, if $(\psi_{\hbar})_{0<\hbar\leq 1}$ is also a sequence of $o(1)$-quasimodes for a certain given Hecke operator $T_p$, then $\mu_0=L$~\cite{BroLin14}. 
\end{rema}

Finally, we define another particular subset of $\ml{M}([\tau_1(\hbar),\tau_2(\hbar)],\hbar\rightarrow 0^+)$, i.e.
$$\ml{M}_{\log}([\tau_1(\hbar),\tau_2(\hbar)],\hbar\rightarrow 0^+),$$
where we impose the extra assumption that the initial data vary among sequences $(\psi_{\hbar})_{0<\hbar\leq 1}$ satisfying 
\begin{equation}\label{e:hosc-log}
\left\|(-\hbar^2\Delta_g-1)\psi_{\hbar}\right\|_{L^2}=o\left(\hbar|\log\hbar|^{-1}\right),\ \|\psi_{\hbar}\|_{L^2}=1.
\end{equation}
According to~\cite{An08}, the semiclassical measures of such initial data have positive metric entropy.

\begin{rema}
We emphasize that \emph{all the sets of semiclassical measures we have introduced in this paragraph will depend implicitly on the choice of the sequence 
$(\eps_{\hbar})_{0<\hbar\leq 1}$, and on the potential $V$}. They both play an important role in the statements below; however, in order to alleviate notations, 
they do not appear explicitely in the conventions we used for the sets of semiclassical measures $\ml{M}(\ldots)$.

\end{rema}

\subsection{Main theorem}

In our different results, an important role will be played by the function
\begin{equation}\label{e:unst-comp}
\forall (x_0,\xi_0)\in S^*M,\ \ f_V(x_0,\xi_0):=g^*_{x_0}(d_{x_0}V,\xi_0^{\perp}),
\end{equation}
where $\xi_0^{\perp}\in S_{x_0}^*M$ is the vector directly orthogonal to $\xi_0$. In the case of negatively curved surfaces, it represents (up to a constant) the unstable component of the Hamiltonian vector field associated to $V$ on $T^*M$~\cite{EsRi14} -- section $3$. More precisely, our results will depend on the geometry of the set of ``critical points'' of $f_V$:
\begin{equation}\label{e:J-critical}
\ml{C}_V:=\bigcap_{j=0}^{+\infty}\left\{(x_0,\xi_0)\in S^*M:(X_0^j.f_V)(x_0,\xi_0)=0\right\},
\end{equation}
where $X_0$ is the geodesic vector field, i.e. $X_0(\rho)=\frac{d}{dt}(G_0^t(\rho))_{|t=0}$ for every $\rho$ in $S^*M$, and where $X_0^j$ means that we differentiate $j$ times in the direction of $X_0$. Our main result is the following:

\begin{theo}\label{t:maintheo} Suppose that $\text{dim}(M)=2$ and that it has constant negative curvature $K\equiv -1$. Suppose that $\eps_{\hbar}\rightarrow 0$ as $\hbar\rightarrow 0$, and that there exists $0<\nu<1/2$ such that, for $\hbar>0$ small enough,
$$\hbar^{\nu}\leq\eps_{\hbar}.$$
 

Then, for every sequence $(\psi_{\hbar_n})_{n\in\mathbb{N}}$ which satisfies~\eqref{e:hosc} with $\hbar_n\rightarrow 0^+$, and which has an unique semiclassical measure $\mu_0$, for every $1<c_1\leq c_2<\min\{3/2, 1/(2\nu)\}$, for every $$\mu\in\ml{M}(\psi_{\hbar_n},[c_1|\log\eps_{\hbar_n}|,c_2|\log\eps_{\hbar_n}|],\hbar_n\rightarrow 0^+),$$ 
and for every $a\in\ml{C}^{\infty}(S^*M,\IR)$, one has
$$\mu_0(\ml{C}_V)\min\{a\}+ (1-\mu_0(\ml{C}_V))\int_{S^*M}adL\leq\mu(a)\leq\mu_0(\ml{C}_V)\max\{a\}+ (1-\mu_0(\ml{C}_V))\int_{S^*M}adL.$$
\end{theo}

This theorem is a consequence of proposition~\ref{p:semicl} below which is slightly more precise as it allows, for each choice of $(\eps_{\hbar})_{0<\hbar\leq 1}$ to consider more general families of initial data -- see remark~\ref{r:general-data}.

In order to clarify our statement, we will give below two corollaries of this theorem. Before, we briefly observe that 
if we choose $V$ satisfying $\mu_0(\ml{C}_V)<1$, then, for every nonempty open subset $\omega$ of $S^*M$, 
one has $\mu(\omega)>0$. In other words, provided the perturbation $V$ satisfies some ``generic'' property with respect to the initial data, then 
the solutions of the perturbed Schr\"odinger equation put mass on every nonempty open subset of $S^*M$ for times of order 
$|\log(\eps_{\hbar})|$. In the case where $\mu(\ml{C}_V)=0$, the limit measure is in fact the Liouville measure. Compared with the quantum ergodicity 
theorem, we emphasize that we do not need to average over the time parameter.

\subsection{Some corollaries}

The following corollary is a direct consequence of the main theorem:

\begin{coro}\label{c:coro1} Suppose that $\text{dim}(M)=2$ and that it has constant negative curvature $K\equiv -1$. Suppose that $\eps_{\hbar}\rightarrow 0$ as $\hbar\rightarrow 0^+$, and that there exists $0<\nu<1/2$ such that for $\hbar>0$ small enough, 
$$\hbar^{\nu}\leq\eps_{\hbar}.$$
Suppose also that
$$\ml{C}_V=\emptyset.$$
Then, for every $1<c_1\leq c_2<\min\{3/2, 1/(2\nu)\}$, one has
$$\ml{M}\left([c_1|\log(\eps_{\hbar})|,c_2|\log(\eps_{\hbar})|]\right)=\{L\}.$$
\end{coro}

\begin{rema}
Observe that the set
$$\ml{C}_V^0:=\left\{(x_0,\xi_0)\in S^*M:f_V(x_0,\xi_0)=0\right\}$$
can never be empty, as the set there always exists $x_0\in M$ such that $d_{x_0}V=0$. In appendix~\ref{a:example}, it is shown that the assumption $\ml{C}_V=\emptyset$ is generic in the sense that it is satisfied on an open and dense subset of $\ml{C}^{\infty}(M,\IR)$.
\end{rema}

In other words, \emph{the solutions of the perturbed Schr\"odinger become equidistributed for time scales of order $|\log(\eps_{\hbar})|$ as 
soon as the potential satisfies some geometric admissibility condition} -- see appendix~\ref{a:example} for the construction of such potentials. This statement is very close to the main theorem of~\cite{EsRi14}. 
In this reference, equidistribution was also shown to hold for times of order 
$|\log(\eps_{\hbar})|$ under perturbation by a ``multi-scaled'' potential which satisfies some geometric admissibility condition. The main 
improvement compared with this reference are the following. First, the geometric assumptions on the potential are simpler, and we do not have to deal 
with a ``multi-scaled perturbation''. Moreover, we do not need to average over subintervals of $[c_1 |\log(\eps_{\hbar})|, c_2|\log(\eps_{\hbar})|]$ to 
obtain equidistribution of the solutions, i.e. equidistribution holds for \emph{any} time in the above interval. Finally, the main result 
from~\cite{EsRi14} holds for a generic $\eps$ in the interval $[0,\hbar^{\nu'}]$, while the present result holds for any large enough $\eps$ -- see proposition~\ref{p:semicl} for a more precise statement. 
Yet, it is important to underline that the restrictions on the energy localization of the initial data are much more restrictive here than 
in~\cite{EsRi14} -- see~\eqref{e:hosc} and~\eqref{e:hosc0}.

As was already mentioned, the main theorem shows that the solutions of the perturbed Schr\"odinger equation put some mass on every nonempty open subset as soon as we know that the initial data do not put all its mass on the set $\ml{C}_V$. The results from~\cite{An08} provide sufficient conditions to ensure this property, modulo the fact that the initial data satisfy the stronger assumption~\eqref{e:hosc-log}. More precisely, we introduce the ``maximal invariant'' subset 
inside $\ml{C}_V$, i.e.
$$\Lambda_V:=\bigcap_{t\in\IR}G_0^t\ml{C}_V,$$
and combining our main result to~\cite{An08}, we obtain the following corollary:

\begin{coro}\label{c:coro2} Suppose that $\text{dim}(M)=2$ and that it has constant negative curvature $K\equiv -1$. Suppose that $\eps_{\hbar}\rightarrow 0$ as $\hbar\rightarrow 0$, and that there exists $0<\nu<1/2$ such 
$$\hbar^{\nu}\leq\eps_{\hbar}.$$
Suppose that the 
topological entropy\footnote{We refer the reader to~\cite{KaHa} for a precise definition.} of $\Lambda_V$ satisfies
$$h_{top}(\Lambda_V)<\frac{1}{2}.$$
Then, for every $1<c_1\leq c_2<\min\{3/2, 1/(2\nu)\}$, for every  $\mu\in\ml{M}_{\log}\left([c_1|\log(\eps_{\hbar})|,c_2|\log(\eps_{\hbar})|]\right)$, and for every nonempty open subset $\omega$ in $S^*M$, one has
$$\mu(\omega)>0.$$
\end{coro}

Thanks to corollary $4$ from~\cite{BaWo07}, one knows that $h_{top}(\Lambda_V)<\frac{1}{2}$ is satisfied as soon as the Hausdorff dimension of $\Lambda_V$ is $<2$. We also recall that, 
$$\Lambda_V\subset\Lambda_V^0:=\bigcap_{t\in\IR}G_0^t\ml{C}_V^0\subset\ml{C}_V^0\subset S^*M,$$ 
and we remark that, for a generic choice of $V$ (say $V$ has finitely many critical points), 
one has $\text{dim}_H(\ml{C}_V^0)=2$.

\subsection{Organization of the article}

In section~\ref{s:geom-back}, we briefly recall some properties of geodesic flows on negatively curved surfaces. In section~\ref{s:semiclassical}, we use semiclassical tools to reduce the proof of theorem~\ref{t:maintheo} to a question on hyperbolic dynamical systems. In section~\ref{s:dynamical}, we solve this dynamical systems question using strong structural stability and unique ergodicity of the horocycle flow. In section~\ref{s:loschmidt}, we apply our method to the study of the quantum Loschmidt echo. In appendix~\ref{a:pdo}, we give a short toolbox on semiclassical analysis, and 
in appendix~\ref{a:stability}, we recall the strong structural stability theorem and we give a brief account on the results from~\cite{EsRi14} that we will use in this article. Finally, appendix~\ref{a:example} provides a large class of potentials satisfying the assumptions of corollary~\ref{c:coro1}.

In all the article, $M$ will denote a smooth ($\ml{C}^{\infty}$), connected, orientable, compact and Riemannian manifold 
without boundary.

\section{Properties of geodesic flows on negatively curved surfaces}
\label{s:geom-back}

In all this article, we make the additional assumption that $M$ is a surface with constant sectional curvature $K\equiv -1$. We will now draw some dynamical consequences of this geometric assumption. We refer to~\cite{KaHa, Rug07} for a more detailed exposition.

In this geometric setting, the geodesic flow\footnote{In this case, $(G_0^t)_{t\in\IR}$ is the Hamiltonian flow associated to $p_0$.} $(G_0^t)_{t\in\IR}$ satisfies the Anosov property 
on $S^*M$~\cite{Ano67}. Precisely, it means that, for every $\rho=(x,\xi)$ in $S^*M$, there exists a $G_0^t$-invariant splitting
\begin{equation}\label{e:anosov}
T_{\rho}S^*M=\IR X_0(\rho)\oplus E^s(\rho)\oplus E^u(\rho),
\end{equation}
where $X_0(\rho)$ is the Hamiltonian vector field associated to $p_0(x,\xi)=\frac{\|\xi\|_x^2}{2}$, $E^u(\rho)$ is the unstable direction and $E^s(\rho)$ is 
the stable direction. These three subspaces are preserved under the geodesic flow and there exist  constants $C_0>0$ and $\gamma_0>0$ such that, for any 
$t\geq 0$, for any $v^s\in E^s(\rho)$ and any $v^u\in E^u(\rho)$,
$$\|d_{\rho}G_0^tv^s\|_{G^t_0(\rho)}\leq C_0e^{-\gamma_0 t}\|v^s\|_{\rho}\ \text{and}\ \|d_{\rho}G_0^{-t}v^u\|_{G^{-t}_0(\rho)}\leq C_0e^{-\gamma_0 t}\|v^u\|_{\rho},$$
where $\|.\|_w$ is the norm associated to the Sasaki metric on $S^*M$~\cite{Rug07}.

Moreover, as $\text{dim}( M)=2$, these three subspaces are $1$-dimensional subspaces of $T_{\rho}S^*M.$ As explained in section $3$ of~\cite{EsRi14}, one can associate a direct orthonormal basis to this splitting that we denote by $(X_0(\rho), X^s(\rho), X^u(\rho))$. 
\begin{rema} We emphasize that we make a slightly different choice of convention compared with~\cite{EsRi14}. In this reference, the stable and unstable vector fields $X^s$ and $X^u$ were chosen to be of norm $\sqrt{2}$. Here, we will use the convention that they are unit vectors for the Sasaki metric on $S^*M$.
\end{rema}
Recall from Chapter $3$ of~\cite{Rug07} (see also section $3$ of~\cite{EsRi14}) that there exists some $C_0>0$ such that, for every $\rho$ in $S^*M$ and for every $t$ in $\IR$,
\begin{equation}\label{e:exp-rate}
\|d_{\rho}G_0^t\|\leq C_0 e^{t}.
\end{equation}

In our proof, we will use two main properties of geodesic flows on negatively curved surfaces, namely
\begin{itemize}
 \item strong structural stability -- see appendix~\ref{a:stability} for a brief reminder of this property~\cite{Ano67, dLLMM86};
 \item unique ergodicity of the horocycle flow~\cite{Fu73, Mar77}.
\end{itemize}
We conclude this preliminary section by a brief reminder on the ergodic properties of horocycle flows on negatively curved surfaces~\cite{Mar77}. Thanks to the fact 
that we are considering negatively curved surfaces, one knows that $X^u$ defines a $\ml{C}^1$ vector field on $S^*M$~\cite{PuHi75}. The unstable horocycle flow is 
then defined as the $\ml{C}^1$ flow $(H^s_{u})_{s\in\IR}$ satisfying
$$\forall\rho\in S^*M,\ \ \frac{d}{ds}\left(H_u^s(\rho)\right)=X_u\circ H^s_u(\rho).$$
Recall from~\cite{Mar77} that this parametrization of the horocycle flow is uniformly expanding in the sense that
$$\forall\rho\in S^*M,\ \forall (t,\tau)\in\IR^2,\ G_0^t\circ H_u^{\tau}(\rho)=H_u^{e^t\tau}\circ G_0^t(\rho).$$
Moreover, this flow is uniquely ergodic~\cite{Fu73, Mar77}:
\begin{theo} Suppose that $\dim(M)=2$, and that $M$ has constant negative sectional curvature $K\equiv -1$. One has, for every $a$ in $\ml{C}^0(S^*M)$,
$$\lim_{T\rightarrow+\infty}\sup\left\{\left|\frac{1}{T}\int_0^T a\circ H^s_u(\rho)ds-\int_{S^*M}adL\right|:\rho\in S^*M\right\}=0,$$
where $L$ is the desintegration of the Liouville measure on $S^*M$. 
\end{theo}
As mentioned above, this property will be at the heart of our proof. In fact, it was shown in~\cite{Mar77} (lemma $4.5$) that this result 
implies that, for every $a$ in $\ml{C}^0(S^*M)$ and every $b>0$,
\begin{equation}\label{e:mixing}
 \lim_{s\rightarrow+\infty}\frac{1}{b}\int_0^b a\circ H_u^s\circ G_0^t(\rho)dt=\int_{S^*M} a dL,
\end{equation}
uniformly in $\rho$. This property exactly says that small pieces of geodesics become equidistributed under the action of the horocycle flow, and it is the central 
step in the proof of strong mixing for horocycle flows given in~\cite{Mar77}. After a reduction based on semiclassical techniques, we will have to 
understand some ergodic properties of perturbed geodesic flows to complete the proof of our main results. This will be done in section~\ref{s:dynamical}, 
and, even if the proof given there does not concern directly the horocycle flow, 
it will be modeled on a similar strategy as the proof of~\eqref{e:mixing} given in~\cite{Mar77}. The main difference is that we will replace $H_u^s$ 
by a perturbation $\tilde{G}_{\eps}^s$ of the geodesic flow and that we will need to prove that, in a certain regime of $b$, $\eps$ and $s$, 
it behaves like the unstable horocycle flow. 

In order to clarify the proof given in section~\ref{s:dynamical}, we briefly recall from~\cite{Mar77} how one can 
derive~\eqref{e:mixing} from the unique ergodicity of the horocycle flow. Let $\eta>0$. It is sufficient to prove that, for $b>0$ small enough, there exists 
$s_0(b,\eta)$ such that, for every $s\geq s_0(b,\eta)$, the average $\frac{1}{b}\int_0^b a\circ H_u^s\circ G_0^t(\rho)$ is within $\eta$ of $\int_{S^*M}adL$.

We write that
$$\frac{1}{b}\int_0^b a\circ H_u^s\circ G_0^t(\rho)dt=\frac{1}{b}\int_0^b a\circ G_0^t\circ H_u^{se^{-t}}(\rho)dt
=\frac{1}{b}\int_0^b a\circ H_u^{se^{-t}}(\rho)dt+\ml{O}(b).$$
Then, we make the change of variables $\tau=se^{-t}$ and we get
$$\frac{1}{b}\int_0^b a\circ H_u^s\circ G_0^t(\rho)dt=-\frac{1}{b}\int_s^{se^{-b}} a\circ H_u^{\tau}(\rho)\frac{d\tau}{\tau}+\ml{O}(b).$$
This quantity looks very much like a Birkhoff average, except that we have a Jacobian factor in the integral. In order to deal with this term, we apply the mean value 
Theorem and we get
$$\frac{1}{b}\int_0^b a\circ H_u^s\circ G_0^t(\rho)dt=-\frac{e^{-b}-1}{be^{-t_0}}\frac{1}{se^{-b}-s}\int_s^{se^{-b}} a\circ H_u^{\tau}(\rho)d\tau+\ml{O}(b),$$
for some $t_0$ in $[0,b]$. We have that $\frac{1-e^{-b}}{be^{-t_0}}=1+\ml{O}(b)$, which implies that
$$\frac{1}{b}\int_0^b a\circ H_u^s\circ G_0^t(\rho)dt=\frac{1}{se^{-b}-s}\int_s^{se^{-b}} a\circ H_u^{\tau}(\rho)d\tau+\ml{O}(b).$$
We can now apply unique ergodicity of the horocycle flow, and we find that there exists a nonincreasing function $r(T)\rightarrow 0$ as $T\rightarrow +\infty$ such that
$$\frac{1}{b}\int_0^b a\circ H_u^s\circ G_0^t(\rho)dt=\int_{S^*M}adL+\ml{O}(b)+r(s(1-e^{-b}))=\int_{S^*M}adL+\ml{O}(b)+o((sb)^{-1}),$$
which implies our result.

\section{Reduction to classical dynamics}
\label{s:semiclassical}

In this section, we consider a slightly more general setting than the one in the introduction in order to allow more general class of initial data. We fix a sequence $(\eps_{\hbar})_{0<\hbar\leq 1}$ which satisfies $\eps_{\hbar}\rightarrow 0$, as $\hbar\rightarrow 0^+$ and which represents the ``strength'' of our perturbation. We define then \emph{admissible} sequences of initial data of order $\nu_0>0$ as follows:
\begin{equation}\label{e:hosc2}
\lim_{\hbar\rightarrow 0^+}\left\Vert \mathbf{1}_{\left[  1-\hbar\eps_{\hbar}^{-\nu_0},1+\hbar\eps_{\hbar}^{-\nu_0}\right]
}\left(  -\hbar^{2}\Delta\right)  \psi_{\hbar}-\psi_{\hbar}\right\Vert _{L^{2}\left(  M\right)
}=0,\ \text{and},\ \forall\ 0<\hbar\leq 1,\ \|\psi_{\hbar}\|_{L^2}=1,
\end{equation}
where $\nu_0>0$ is some positive constant.

\begin{rema}\label{r:general-data}
For any choice of $(\eps_{\hbar})_{0<\hbar\leq 1}$, assumption~\eqref{e:hosc2} allows to consider sequences of normalized initial data satisfying 
\begin{equation}\label{e:Rh-initial}
\limsup_{\hbar\rightarrow 0^+}\left\Vert \mathbf{1}_{\left[  1-R\hbar,1+R\hbar\right]
}\left(  -\hbar^{2}\Delta\right)  \psi_{\hbar}-\psi_{\hbar}\right\Vert _{L^{2}\left(  M\right)
}\longrightarrow 0,\ \ \ \text{as}\ \ \ \ \ R\rightarrow+\infty.\end{equation}
We observe that this remark combined to proposition~\ref{p:semicl} below allows to deduce our main statement, i.e. theorem~\ref{t:maintheo}.

\end{rema}

For every $\eps\in[0,\eps_{\hbar}]$, we introduce the following self-adjoint deformation of the Schr\"odinger operator:
$$\hat{P}_{\eps}(\hbar):=-\frac{\hbar^2\Delta_g}{2}+\eps V,$$
where $V$ belongs to $\ml{C}^{\infty}(M,\IR)$. We will denote by $(G_{\eps}^t)_{t\in\IR}$ the Hamiltonian flow associated to $p_{\eps}(x,\xi):=p_0(x,\xi)+\eps V(x)$.

Our goal is to study the long time dynamics of the corresponding quantum propagator 
$U_{\hbar,\eps}(t):=e^{-\frac{i t\hat{P}_{\eps}(\hbar)}{\hbar}}$ on initial data satisfying~\eqref{e:hosc2}, and for $\eps$ large enough belonging to
$[0,\eps_{\hbar}]$. For that purpose, we introduce a sequence of time scales $(\tau_{\hbar})_{0<\hbar\leq 1}$ which satisfies $\tau_{\hbar}\rightarrow +\infty$ as $\hbar\rightarrow 0^+$.

\begin{rema} In the following, our results will concern scales of times $\tau_{\hbar}$ of order $|\log(\eps_{\hbar})|$, under the assumption that $\eps_{\hbar}\gg\sqrt{\hbar}$. In our geometric context, it exactly means that we consider a regime where the semiclassical approximation is valid, i.e. below the Ehrenfest time. In fact, in our setting, the Ehrenfest time in the Egorov theorem is known to be of order $|\log\hbar|/2$~\cite{AnNo07, DyGu14} -- see appendix~\ref{a:pdo} for a brief reminder. 
\end{rema}

Given a sequence of normalized initial data $(\psi_{\hbar})_{0<\hbar\leq 1}$ satisfying~\eqref{e:hosc2}, we define
$$u_{\hbar}^{\eps}(\tau_{\hbar}):=e^{-\frac{i \tau_{\hbar}\hat{P}_{\eps}(\hbar)}{\hbar}}\psi_{\hbar},$$
and its ``Wigner distribution'' on $T^*M$, i.e.
\begin{equation}\label{e:wigner-eps}\forall a\in\ml{C}^{\infty}_c(T^*M),\ \mu_{\hbar}^{\eps}(\tau_{\hbar})(a):=\left\la u_{\hbar}^{\eps}(\tau_{\hbar}),
\Oph(a)u_{\hbar}^{\eps}(\tau_{\hbar})\right\ra_{L^2} 
\end{equation}
The goal of this section is to describe the asymptotic properties of $\mu_{\hbar}^{\eps}(\tau_{\hbar})$ as $\hbar\rightarrow 0^+$, $\tau_{\hbar}\sim|\log(\eps_{\hbar})|$, and $\eps\in[0,\eps_{\hbar}]$ large enough.

\begin{rema}\label{r:sc-measure-initial}
The properties we will obtain will depend on the choice of $V$ and on the properties of the semiclassical measures of the 
initial data.

In all of this section, we will choose $\eps_{\hbar}$ and $\nu_0$ in such a way that $\hbar\eps_{\hbar}^{-\nu_0}\rightarrow 0$ as 
$\hbar\rightarrow 0^+$.  Recall then that, one can extract a subsequence $\hbar_n\rightarrow 0$ such that the sequence of distributions $(\mu_{\hbar_n}^0(0))_{n\in\mathbb{N}}$ converges to a limit distribution which is in fact a probability measure $\mu_0$ carried on the unit cotangent bundle $S^*M$~\cite{Zw12} -- Chapter~$5$. We underline that $\mu_0$ does not have a priori extra properties like invariance by the geodesic flow, even under the stronger assumption~\eqref{e:Rh-initial}.

We will only consider sequences of initial data with an unique semiclassical measure $\mu_0$ and, in order to alleviate the notations, we denote the extraction $\hbar_n\rightarrow 0$ by $\hbar\rightarrow 0^+$.

\end{rema}

Finally, the properties we will obtain are related to the set of critical points of order $J\geq 0$, i.e.
\begin{equation}\label{e:J-critical-points}
\ml{C}_V^J:=\bigcap_{j=0}^{J}\left\{(x_0,\xi_0)\in S^*M:(X_0^j.f_V)(x_0,\xi_0)=0\right\}.
\end{equation}

\subsection{Statement of the main result}

Before stating our main result, we start with a simple observation which follows directly from the results described in paragraph~\ref{ss:antiwick}. Without any 
assumptions on $(\tau_{\hbar})_{0<\hbar\leq 1}$, on $(\eps_{\hbar})_{0<\hbar\leq 1}$ (except that $\eps_{\hbar}\rightarrow 0$) and on 
the geometry of the manifold, we always have
\begin{equation}\label{e:garding}\min_{S^*M}\{a\}
\leq\liminf_{\hbar\rightarrow 0^+,\eps\in[0,\eps_{\hbar}]}\mu_{\hbar}^{\eps}(\tau_{\hbar})(a)
\leq\limsup_{\hbar\rightarrow 0^+,\eps\in[0,\eps_{\hbar}]}\mu_{\hbar}^{\eps}(\tau_{\hbar})(a)
\leq \max_{S^*M}\{a\}.\end{equation}
The proposition below will show that, under proper assumptions on the geometry of the manifold and on $V$, and for strong enough perturbations $(\eps_{\hbar})_{0<\hbar\leq 1}$, one can find a scale of times $(\tau_{\hbar})_{0<\hbar\leq 1}$ for which $(\mu_{\hbar}^{\eps}(\tau_{\hbar})(a))_{\hbar\rightarrow 0^+}$ 
asymptotically belong to a smaller interval. More precisely, we will show

\begin{prop}\label{p:semicl} Suppose that $\dim(M)=2$, and that $M$ has constant negative sectional curvature $K\equiv -1$. Suppose that $\lim_{\hbar\rightarrow 0}\eps_{\hbar}=0$, and that there exists $0<\nu<\frac{1}{2}$ such that, for $\hbar>0$ small enough, one has
$$\eps_{\hbar}\geq \hbar^{\nu}.$$
Let $J$ be a nonnegative integer. Let $\nu_0>0$ and $\nu_1\geq 0$ satisfying
$$1+\nu_1+(3J+1)\nu_0<\min\left\{\frac{3}{2},\frac{1}{2\nu}\right\}.$$
Suppose that there exists $\displaystyle 1+\nu_1+(3J+1)\nu_0<c_1\leq c_2<\min\left\{\frac{3}{2},\frac{1}{2\nu}\right\}$ such that, for $\hbar>0$ small enough, 
$$c_1|\log(\eps_{\hbar})|\leq\tau_{\hbar}\leq c_2|\log(\eps_{\hbar})|.$$

Then, for any normalized sequence $(\psi_{\hbar})_{\hbar\rightarrow 0^+}$ in $L^2(M)$ which satisfies~\eqref{e:hosc2} with $\nu_0$, and which has an unique semiclassical measure $\mu_0$, one has, for every $a$ in $\ml{C}^{\infty}_c(T^*M,\IR)$,

 $$\mu_0(\ml{C}_V^J)\min_{S^*M}\{a\}+(1-\mu_0(\ml{C}_V^J))\int_{S^*M}adL
\leq\liminf_{\hbar\rightarrow 0^+,\eps\in[\eps_{\hbar}^{1+\nu_1},\eps_{\hbar}]}\mu_{\hbar}^{\eps}(\tau_{\hbar})(a)$$
$$\hspace{3cm}\leq\limsup_{\hbar\rightarrow 0^+,\eps\in[\eps_{\hbar}^{1+\nu_1},\eps_{\hbar}]}\mu_{\hbar}^{\eps}(\tau_{\hbar})(a)
\leq \mu_0(\ml{C}_V^J)\max_{S^*M}\{a\}+(1-\mu_0(\ml{C}_V^J))\int_{S^*M}adL.$$
\end{prop}

\begin{rema}
This result implies Theorem~\ref{t:maintheo} when we restrict ourselves to sequences of initial data satisfying~\eqref{e:Rh-initial} -- see remark~\ref{r:general-data}. 
\end{rema}

\begin{rema}\label{r:homogen-0} Our assumptions on the different parameters impose that $0<\nu_0<\frac{1}{2(3J+1)}\leq\frac{1}{2}$ and $\eps_{\hbar}\geq\sqrt{\hbar}$ asymptotically. In particular, one has $\hbar\eps_{\hbar}^{-\nu_0}\rightarrow 0$ when $\hbar\rightarrow 0^+$, as in the context of remark~\ref{r:sc-measure-initial}.
\end{rema}

The statement of the proposition is a little bit technical as it involves many parameters, and it should be understood as follows. If we suppose that $\eps_{\hbar}\gg\sqrt{\hbar}$, then one can find admissibility conditions on frequencies of the initial data, and a scale of times for which the sequence $(\mu_{\hbar}^{\eps}(\tau_{\hbar})(a))_{\hbar,\eps}$ belongs asymptotically to the interval
$$I_{\mu_0,V,J}(a):=\left[\mu_0(\ml{C}_V^J)\min_{S^*M}\{a\}+(1-\mu_0(\ml{C}_V^J))\int_{S^*M}adL,\mu_0(\ml{C}_V^J)\max_{S^*M}\{a\}+(1-\mu_0(\ml{C}_V^J))\int_{S^*M}adL\right],$$
which is a subinterval of $[\min a,\max a]$ that appeared in~\eqref{e:garding}. The interval is smaller as soon as $\mu_0$, $V$ and $J$ satisfy $\mu_0(\ml{C}_V^J)<1$. In the case where they satisfy $\mu_0(\ml{C}_V^J)=0$, then the matrix elements converge in fact to $\int_{S^*M}adL.$ 

\begin{rema}
We underline that, if we fix a normalized sequence $(\psi_{\hbar})_{\hbar\rightarrow 0^+}$ with an unique semiclassical measure $\mu_0$, then one has $\mu_0(\ml{C}_V^{J+1})\leq\mu_0(\ml{C}_V^J)$ and thus $I_{\mu_0,V,J+1}(a)\subset I_{\mu_0,V,J}(a)$. Yet, one has to be careful as the condition on $(\psi_{\hbar})_{\hbar\rightarrow 0^+}$ becomes more restrictive as we increase the parameter $J$. For instance, given a normalized sequence $(\psi_{\hbar})_{\hbar\rightarrow 0^+}$, it could happen that the proposition could be applied for some $J$ but not necessarly for $J+1$. 
\end{rema}

This result is very close to theorem $4.1$ from~\cite{EsRi14}, and the main improvement compared with this reference is that \emph{we do not need to average over the perturbation parameter} $\eps\in[0,\eps_{\hbar}]$ in order to get an equidistribution property. However, we have to make a much stronger restriction on the family of initial data as it was only required in~\cite{EsRi14} that
$$\forall\delta_0>0,\ \lim_{\hbar\rightarrow 0^+}\left\Vert \mathbf{1}_{\left[  1-\delta_0,1+\delta_0\right]
}\left(  -\hbar^{2}\Delta\right)  \psi_{\hbar}-\psi_{\hbar}\right\Vert _{L^{2}\left(  M\right)
}=0,$$
which is obviously a much weaker assumption than~\eqref{e:hosc2}.

\subsection{Preliminary observations}
\label{sss:homogen}

Thanks to the frequency assumption~\eqref{e:hosc2} and to Remark~\ref{r:homogen-0}, we can suppose without loss of generality that $a$ is homogeneous in a 
neighborhood of $S^*M$, i.e. there exists $0<\delta_0<1/2$ such that
\begin{equation}\label{e:homogeneity}
\forall (x,\xi)\ \text{satisfying}\ p_0(x,\xi)\in[1/2-\delta_0,1/2+\delta_0],\ a(x,\xi)=a\left(x,\frac{\xi}{\|\xi\|}\right).
\end{equation}
Without loss of generality, we can also suppose that
\begin{equation}\label{e:quasimode}
 \mathbf{1}_{\left[  1-\eps_{\hbar}^{-\nu_0}\hbar,1+\eps_{\hbar}^{-\nu_0}\hbar\right]}\left(  -\hbar^{2}\Delta\right)  \psi_{\hbar}=\psi_{\hbar}.
\end{equation}
As mentionned in remark~\ref{r:sc-measure-initial}, the semiclassical measure $\mu_0$ has a priori no invariance properties under such general assumptions. Still, we can observe invariance for very short scales of times in $\hbar$. In fact, according to~\eqref{e:quasimode}, one has
\begin{equation}\label{e:evol-quasimode}
e^{-\frac{it(-\hbar^2\Delta_g-1)}{2\hbar}}\psi_{\hbar}=\psi_{\hbar}+\ml{O}(|t|\eps_{\hbar}^{-\nu_0}).
\end{equation}
In particular, for every $\nu_2>\nu_0$, one has
\begin{equation}\label{e:ev-qm}
\forall |t|\leq\eps_{\hbar}^{\nu_2},\ e^{\frac{it\hbar\Delta_g}{2}}\psi_{\hbar}=e^{-\frac{it}{2\hbar}}\psi_{\hbar}+o(1),
\end{equation}
where the remainder is uniform for $t$ in this interval.

\begin{rema} Equality~\eqref{e:ev-qm} is crucial in our proof of proposition~\ref{p:semicl}. In fact, our argument will make use of an equidistribution result, and we will need to average over some parameter in order to use this equidistribution property. In~\cite{EsRi14}, the averaging was performed over the perturbation parameter $\eps\in[0,\eps_{\hbar}]$. Here, we will take advantage of the fact that the initial data satisfies a stronger spectral localization, and we will use~\eqref{e:ev-qm} in order to average over a time parameter $t\in[0,\eps_{\hbar}^{\nu_2}]$.
 
\end{rema}

\subsection{Proof of proposition~\ref{p:semicl}} 

\label{sss:proof-semicl}

The proof of this proposition can be divided in two main steps: $(1)$ we apply semiclassical 
rules in order to reduce ourselves to a question on ergodic properties of geodesic flows; $(2)$ we use tools from hyperbolic dynamical systems in order to answer this ``dynamical systems'' question. This paragraph is devoted to the first step, and the proof of the second step is postponed to section~\ref{s:dynamical}.

Let $J$, $\nu$, $\nu_0$, $\nu_1$, $c_1$ and $c_2$ be as in the statement of proposition~\ref{p:semicl}. Recall that $\eps_{\hbar}\geq\hbar^{\nu}$ and that$$c_1|\log(\eps_{\hbar})|\leq\tau_{\hbar}\leq c_2|\log(\eps_{\hbar})|,$$
for $\hbar>0$ small enough. Let $a$ be an element in $\ml{C}^{\infty}_c(T^*M,\IR)$ as in paragraph~\ref{sss:homogen}, i.e. which is $0$-homogeneous in a neighborhood of size $\delta_0$ of $S^*M$. We underline that it is sufficient to prove the lower bound as the upper bound can then be obtained by replacing $a$ by $-a$.

\subsubsection{Truncation in phase space}

Let $0<\delta<\delta_0/4$. We introduce $0\leq \chi_{\delta}\leq 1$ a smooth function on $\IR$ which is equal to $1$ on the interval $[(1-\delta)/2,(1+\delta)/2]$ and $0$ outside some interval $[1/2-\delta,1/2+\delta]$. Thanks to~\eqref{e:quasimode}, 
we can write
$$\psi_{\hbar}=\chi_{\delta}(\hat{P}_0(\hbar))\psi_{\hbar}+o(1).$$
Recall that the operator $\chi_{\delta}(\hat{P}_{\eps}(\hbar))$ is a $\hbar$-pseudodifferential operator in $\Psi^{-\infty,0}(M)$ with principal symbol $\chi_{\delta}\circ p_{\eps}(x,\xi)$~\cite{Zw12} (Ch.~$14$) and with the semi-norms which are uniformly bounded for $\eps\in[0,\eps_{\hbar}]$. Then, thanks to the Calder\'on-Vaillancourt theorem~\cite{Zw12}, we find that, uniformly for 
$\eps\in[0,\eps_{\hbar}]$,
$$\psi_{\hbar}=\chi_{\delta}(\hat{P}_{\eps}(\hbar))\psi_{\hbar}+o(1).$$
This implies that, one has
$$\mu_{\hbar}^{\eps}(\tau_{\hbar})(a)=\left\la\psi_{\hbar},
e^{+\frac{i\tau_{\hbar}\hat{P}_{\eps}(\hbar)}{\hbar}}\Oph(a)\chi_{\delta}(\hat{P}_{\eps}(\hbar))e^{-\frac{i\tau_{\hbar}\hat{P}_{\eps}(\hbar)}{\hbar}}\psi_{\hbar}\right\ra 
+o(1),$$
where the remainder is uniform for $\eps\in[0,\eps_{\hbar}]$. We now apply the composition formula for pseudodifferential operators and we find that
\begin{equation}\label{e:step1}\mu_{\hbar}^{\eps}(\tau_{\hbar})(a)=\left\la\psi_{\hbar},
e^{+\frac{it\tau_{\hbar}\hat{P}_{\eps}(\hbar)}{\hbar}}\Oph(a\times\chi_{\delta}\circ p_{\eps})e^{-\frac{it\tau_{\hbar}\hat{P}_{\eps}(\hbar)}{\hbar}}\psi_{\hbar}\right\ra 
+o(1),
\end{equation}
where the remainder is still uniform for $\eps\in[0,\eps_{\hbar}]$.

\subsubsection{Long time Egorov property}

Observe now that, for $\hbar>0$ small enough and $\eps\in[0,\eps_{\hbar}]$, the function $a\times\chi_{\delta}\circ p_{\eps}$ is compactly supported in the energy layer $\{(x,\xi):1/2-2\delta\leq|\xi|^2/2\leq 1/2+2\delta\}$. Recall also that 
$$0\leq \tau_{\hbar}\leq c_2|\log(\eps_{\hbar})|\leq c_2\nu|\log\hbar|,$$
with $c_2\nu<1/2$. Thus, we can choose $0<\delta<\delta_0/4$ small enough in a way that depends only 
on $c_2\nu  \in [0,1/2)$ and such that we can apply Egorov Theorem up to the time $\tau_{\hbar}$ (uniformly for $\eps\in[0,\eps_{\hbar}]$ and $\tau_{\hbar}\leq c_2|\log(\eps_{\hbar})|$) -- see~\eqref{e:large-time-egorov} in the appendix. In other words, we have that, for $0<\delta<\delta_0/4$ small enough, 
$$\mu_{\hbar}^{\eps}(\tau_{\hbar})(a)=\left\la\psi_{\hbar},
\Oph(a\circ G_{\eps}^{\tau_{\hbar}}\times\chi_{\delta}\circ p_{\eps})\psi_{\hbar}\right\ra+o(1),$$
where the remainder is uniform for $\eps\in[0,\eps_{\hbar}]$ and 
$0\leq \tau_{\hbar}\leq c_2|\log(\eps_{\hbar})|$.

\subsubsection{Invariance for short times}

We now use~\eqref{e:ev-qm}, i.e. invariance by the free Schr\"odinger equation over short intervals of times. More precisely, we write that one has
\begin{equation}\label{e:step2}\mu_{\hbar}^{\eps}(\tau_{\hbar})(a)=\frac{1}{\eps_{\hbar}^{\nu_2}}\int_0^{\eps_{\hbar}^{\nu_2}}\left\la\psi_{\hbar},
e^{-\frac{is\hbar\Delta_g}{2}}\Oph(a\circ G_{\eps}^{\tau_{\hbar}}\times\chi_{\delta}\circ p_{\eps})e^{\frac{is\hbar\Delta_g}{2}}\psi_{\hbar}\right\ra ds+o(1),\end{equation}
where $\nu_2>\nu_0$. As discussed in appendix~\ref{a:pdo}, the symbol $a\circ G_{\eps}^{\tau_{\hbar}}\times\chi_{\delta}\circ p_{\eps}$ remains in a class of symbols 
$S^{-\infty,0}_{\nu'}(T^*M)$ with $0\leq c_2\nu<\nu'<1/2$ (with the semi-norms which are uniformly bounded for $\tau_{\hbar}\leq c_2|\log(\eps_{\hbar})|$ 
and $\eps\in[0,\eps_{\hbar}]$). In particular, we can apply Egorov theorem for finite time (here $\eps_{\hbar}^{\nu_2}\rightarrow 0$), and we find that
$$\mu_{\hbar}^{\eps}(\tau_{\hbar})(a)=\left\la\psi_{\hbar},
\Oph\left(\frac{1}{\eps_{\hbar}^{\nu_2}}\int_0^{\eps_{\hbar}^{\nu_2}}a\circ G_{\eps}^{\tau_{\hbar}}\circ G_0^s
\times\chi_{\delta}\circ p_{\eps}\circ G_0^sds\right)\psi_{\hbar}\right\ra +o(1).$$
Using Calder\'on-Vaillancourt Theorem one more time, one gets
$$\mu_{\hbar}^{\eps}(\tau_{\hbar})(a)=\left\la\psi_{\hbar},
\Oph\left(\chi_{\delta}\circ p_{0}\times\frac{1}{\eps_{\hbar}^{\nu_2}}\int_0^{\eps_{\hbar}^{\nu_2}}a\circ G_{\eps}^{\tau_{\hbar}}\circ G_0^s
ds\right)\psi_{\hbar}\right\ra +o(1).$$

\subsubsection{Using ergodic properties of the classical flow}
As all the symbols are compactly supported and as they belong to an admissible class of symbols $S^{-\infty,0}_{\nu'}(T^*M)$ with $0\leq c_2\nu<\nu'<1/2$, we can use the results from paragraph~\ref{ss:antiwick}. It means that wa can replace $\Oph$ by a positive quantization $\Oph^+$ (see~\eqref{e:positive-quantization} for instance), i.e.
$$\mu_{\hbar}^{\eps}(\tau_{\hbar})(a)=\left\la\psi_{\hbar},
\Oph^+\left(\chi_{\delta}\circ p_{0}\times\frac{1}{\eps_{\hbar}^{\nu_2}}\int_0^{\eps_{\hbar}^{\nu_2}}a\circ G_{\eps}^{\tau_{\hbar}}\circ G_0^s
ds\right)\psi_{\hbar}\right\ra +o(1).$$
Fix now $\eta_0>0$ and introduce the following subset of $S^*M$:
$$K_V^J(\eta_0):=\left\{(x,\xi)\in S^*M:\exists 0\leq j\leq J,\ \left|(X_0-\text{Id})^j.f_V\right|\left(x,\xi\right)\geq \eta_0\right\},$$
where we recall that $f_V(x,\xi):=g_x^*(d_xV,\xi^{\perp}).$ We also define the following compact subset of $T^*M$
$$\tilde{K}_V^J(\eta_0,\delta):=\left\{(x,\xi)\in T^*M:\left(x,\frac{\xi}{\|\xi\|}\right)\in K_V^J(\eta_0)\ \text{and}\ p_0(x,\xi)\in[1/2-\delta,1/2+\delta]\right\}.$$
\begin{rema}
 The points in the set $\tilde{K}_V^J(\eta_0,\delta)$ corresponds to the points $(x,\xi)$ in $T^*M$ for which we are able to control \emph{uniformly} the convergence of the integral
$$\frac{1}{\eps_{\hbar}^{\nu_2}}\int_0^{\eps_{\hbar}^{\nu_2}}a\circ G_{\eps}^{\tau_{\hbar}}\circ G_0^s(x,\xi)
ds.$$
We refer to proposition~\ref{p:dynamics} for a precise statement.
\end{rema}

We introduce a smooth cutoff function $\chi_V^{\eta_0}$ which is identically equal to $1$ on $\tilde{K}_V^J(\eta_0,\delta)$ and which vanishes outside 
$\tilde{K}_V^J(\eta_0/2,\delta)$. We define
\begin{equation}\label{e:lower-bound-good-set}A_{\hbar,\eta_0,\delta}(\eps):=\inf\left\{\frac{1}{\eps_{\hbar}^{\nu_2}}\int_0^{\eps_{\hbar}^{\nu_2}}a\circ G_{\eps}^{\tau_{\hbar}}\circ G_0^s(x,\xi)
ds:(x,\xi)\in \tilde{K}_V^J(\eta_0/2,\delta)\right\},\end{equation}
Using these notations and the positivity of $\Oph^+$, we derive that
$$\min_{S^*M}\{a\}\mu_{\hbar}^0(0)(\chi_{\delta}\circ p_{0}(1-\chi_V^{\eta_0}))+A_{\hbar,\eta_0,\delta}(\eps)\mu_{\hbar}^0(0)(\chi_{\delta}\circ p_{0}\chi_V^{\eta_0})
\leq\mu_{\hbar}^{\eps}(\tau_{\hbar})(a).$$
Recall that $a$ is homogeneous in a neighborhood of size $\delta_0>0$ of $S^*M$. Thus, as $G_{\eps}^{\tau_{\hbar}}\circ G_0^s(x,\xi)$ remains in this neighborhood 
for all $(t,s)$ when $1/2-\delta\leq p_0(x,\xi)\leq 1/2+\delta$ (provided we choose $\eps\geq0$ small enough), we can replace $a$ by $\tilde{a}$ in the definition of 
$A_{\hbar,\eta_0,\delta}$, where $\tilde{a}(x,\xi):=a(x,\xi/\|\xi\|)$ for every $(x,\xi)$ in $T^*M-M$. In particular, provided we pick $0<\delta<\delta_0/4$ small enough and $\eps\in[\eps_{\hbar}^{1+\nu_1},\eps_{\hbar}]$, 
we can apply proposition~\ref{p:dynamics} which implies that $A_{\hbar,\eta_0,\delta}(\eps)$ converges to $\int_{S^*M} adL$.

We now take the limit $\hbar\rightarrow 0^+$, and we deduce that
$$\mu_0(K_V^{J}(\eta_0/2)^c)\min_{S^*M}\{a\}+\mu_0(K_V^{J}(\eta_0))\int_{S^*M}adL
\leq\liminf_{\hbar\rightarrow 0,\eps\in[\eps_{\hbar}^{1+\nu_1},\eps_{\hbar}]}\mu_{\hbar}^{\eps}(\tau_{\hbar})(a).$$
This property holds for any $\eta_0>0$. Thus, we finally derive
$$\mu_0(\ml{C}_V^J)\min_{S^*M}\{a\}+\mu_0((\ml{C}_V^J)^c)\int_{S^*M}adL
\leq\liminf_{\hbar\rightarrow 0,\eps\in[\eps_{\hbar}^{1+\nu_1},\eps_{\hbar}]}\mu_{\hbar}^{\eps}(\tau_{\hbar})(a),$$
which concludes the proof of proposition~\ref{p:semicl}.

\section{Perturbations of the geodesic flow}
\label{s:dynamical}

Thanks to the results of section~\ref{s:semiclassical}, the proof of our main result is now reduced to a purely dynamical systems question as it only remains to estimate the quantity $A_{\hbar,\eta_0,\delta}(\eps)$ defined by~\eqref{e:lower-bound-good-set}. Precisely, for a fixed 
$0$-homogeneous $\ml{C}^1$ function $\tilde{a}$ on $T^*M-M$, we need to understand the asymptotic behaviour of
$$I_{x_0,\xi_0}(\eps, b, T):=\frac{1}{b}\int_0^{b}\tilde{a}\circ G_{\eps}^{T}\circ G_0^s(x_0,\xi_0)ds,$$
as $b,\eps\rightarrow 0$ and $T\rightarrow+\infty$. Recall that $(G_{\eps}^t)_{t\in\IR}$ is the Hamiltonian flow associated to the function 
$p_{\eps}(x,\xi):=\frac{\|\xi\|^2_x}{2}+\eps V(x)$. This integral looks very much like the integral involved in~\eqref{e:mixing} except that the unstable horocycle has been 
replaced by a perturbed geodesic flow. The way we will deal with the convergence of $I_{x_0,\xi_0}(\eps, b,T)$ will in fact be very similar to the proof of~\eqref{e:mixing} 
given in section~\ref{s:geom-back}. The additional arguments we will need will be:
\begin{itemize}
 \item the strong structural stability theorem~\cite{Ano67, dLLMM86} which will allow us to ``replace'' the perturbed geodesic flow by a ``reparametrization'' of the horocycle flow;
 \item a theorem due to Cartan~\cite{Ca28} on polynomials which helps us to estimate the size of the Jacobian factor.
\end{itemize}
Fix now $J\geq 0$ and $\eta_0>0$. In order to give our main result on the convergence of $I_{x_0,\xi_0}(\eps, b, T)$, recall that we defined the following subset of $S^*M$:
\begin{equation}\label{e:good-points}K_V^J(\eta_0):=\left\{(x,\xi)\in S^*M:\exists 0\leq j\leq J,\ \left|(X_0-\text{Id})^j.f_V\right|\left(x,\xi\right)\geq \eta_0\right\},\end{equation}
where we set $f_V(x,\xi):=g_x^*(d_xV,\xi^{\perp}).$ Using this convention, we have the following statement:

\begin{prop}\label{p:dynamics}Suppose that $\dim(M)=2$, and that $M$ has constant negative sectional curvature $K\equiv -1$. Let $J\geq 0$, and let $\eta_0>0$. Let 
$\nu_1+(3J+1)\nu_2<1/2$ with $\nu_1\geq 0$ and $\nu_2>0$. Let 
$$1+\nu_1+(3J+1)\nu_2<c_1\leq c_2<\frac{3}{2}.$$
Then, there exists $\delta_1>0$ such that, for every $\ml{C}^1$ function $\tilde{a}$ on $T^*M-M$ which is $0$-homogeneous, one has
$$\lim_{\eps_0\rightarrow 0}\sup_{(*)}\left\{\left|\frac{1}{\eps_0^{\nu_2}}\int_0^{\eps_0^{\nu_2}}\tilde{a}\circ G_{\eps}^{c|\log\eps_0|}\circ G_0^s(x_0,\xi_0)ds
-\int_{S^*M}\tilde{a}dL\right|\right\}=0,$$
where $(*)$ means that we take the supremum over the set
$$\left\{(x_0,\xi_0,c,\eps)\in T_{[1/2-\delta_1,1/2+\delta_1]}^*M\times [c_1,c_2]\times [\eps_0^{1+\nu_1},\eps_0]:\left(x_0,\frac{\xi_0}{\|\xi_0\|}\right)\in K_V^J(\eta_0)\right\},$$
with $T_{[1/2-\delta_1,1/2+\delta_1]}^*M :=\{(x_0,\xi_0)\in T^*M: 1/2-\delta_1\leq p_0(x_0,\xi_0)\leq 1/2+\delta_1\}.$
\end{prop}

This proposition tells us that small pieces of geodesics become equidistributed under the action of a perturbed geodesic flow provided that the perturbation is 
nontrivial on the small piece we consider. The fact that the perturbation is nontrivial is exactly guaranteed by the fact that we require $(x_0,\xi_0/\|\xi_0\|)$ to be on 
the subset $K_V^J(\eta_0)$. We emphasize that this statement looks very much like the results\footnote{We emphasize that the results are not equivalent, and that they cannot 
a priori be deduced one from the other.} of section~$6$ in~\cite{EsRi14}, more precisely corollary $6.4$. In this reference, instead of averaging over the time parameter $s$, the average was performed over the perturbation parameter $\eps$. Then, the result was established 
as an equidistribution property for similar scales of times, and the admissibility condition on the perturbation involved the nonvanishing of the following integral 
transform:
$$\forall\rho_0=(x_0,\xi_0)\in S^*M,\ \beta^u_V(x_0,\xi_0):=\frac{1}{\sqrt{2}}\int_0^{+\infty}g_{x(\tau)}^*(d_{x(\tau)}V,\xi^{\perp}(\tau))e^{-\tau}d\tau,$$
where $(x(\tau),\xi(\tau)):=G_0^{\tau}(x_0,\xi_0)$. We refer to appendix~\ref{a:stability} to see how this transform appears naturally when we apply the strong structural stability 
theorem.

This function will also play an important role in our proof and our admissibility condition will in fact be related to it. According to~\cite{EsRi14}, this 
function is H\"older continuous for every\footnote{The constant $1/2$ appearing here is the main reason for the factor $3/2(=1+1/2)$ involved in the statement of the 
proposition.} $\gamma<1/2$. Yet, we have more regularity if we look at the direction of the geodesic flow. In fact, one has
$$e^{-s}\beta_V^u\circ G_0^{s}\left(x_0,\xi_0\right)=\beta_V^u\left(x_0,\xi_0\right)-\int_0^{s}e^{-\tau}f_V\circ G_0^{\tau}\left(x_0,\xi_0\right)d\tau,$$
and we can then observe that the map $s\mapsto e^{-s}\beta_V^u\circ G_0^{s}\left(x_0,\xi_0\right)$ is of class $\ml{C}^{\infty}$ (for a fixed choice of $(x_0,\xi_0)$). 
Moreover, the quantities appearing in the definition of $K_V^J(\eta_0)$ are exactly the derivatives of this map at $s=0$.

We describe now more precisely the main stages of the proof:
\begin{enumerate}
 \item in paragraph~\ref{ss:projection}, using homogeneity properties of our problem, we ``project'' everything on $S^*M$;
 \item in paragraph~\ref{ss:stability}, we use the strong structural stability theorem to replace the perturbed geodesic flow by a ``reparametrized'' 
horocycle flow involving the derivatives of the map $s\mapsto e^{-s}\beta_V^u\circ G_0^{s}\left(x_0,\xi_0\right)$;
 \item in paragraph~\ref{ss:unique-erg}, we make use of the unique ergodicity of the horocycle flow to conclude.
\end{enumerate}

The main lines of the proof are very close to the arguments given in section~$6$ of~\cite{EsRi14}; yet, some aspects need a slightly different treatment, especially in steps $(2)$ and $(3)$.

\subsection{Reduction to $S^*M$}\label{ss:projection}

As was already explained, we will first ``project'' on $S^*M$ all the quantities involved in the definition of  $I_{x_0,\xi_0}(\eps, b, T)$. We follow the same procedure 
as in~\cite{EsRi14} (paragraph $5.1$) and we refer to it for the details.

Let $(x_1,\xi_1)$ be an element in a small neighborhood of $S^*M$ and define
$$\Sigma_{x_1,\xi_1}^{\eps}:=\{(x,\xi)\in T^*M:p_{\eps}(x,\xi)=p_{\eps}(x_1,\xi_1)\},$$ 
which is an energy layer for the Hamiltonian $p_{\eps}$. Introduce also the two following diffeomorphisms:
$$\theta^{\eps}_{x_1,\xi_1}:\Sigma_{x_1,\xi_1}^{\eps}\rightarrow S^*M,\ (x,\xi)\mapsto (x,\xi/\|\xi\|),$$
and its inverse
$$\left(\theta^{\eps}_{x_1,\xi_1}\right)^{-1}:S^*M\rightarrow \Sigma_{x_1,\xi_1}^{\eps},\ (x,\xi)\mapsto \left(x,\sqrt{2(p_{\eps}(x_1,\xi_1)-\eps V(x))}\xi\right).$$
Thanks to these two maps, we can define a new flow on $S^*M$, i.e.
$$\varphi_{\eps,x_1,\xi_1}^t=\theta^{\eps}_{x_1,\xi_1}\circ G_{\eps}^{t/\|\xi_1\|}\circ \left(\theta^{\eps}_{x_1,\xi_1}\right)^{-1}.$$
Recall that we can compute explicitely the vector field
\begin{equation}
Y_{x_1,\xi_1}^{\eps}(\rho):=\frac{d}{dt}\left(\varphi_{\eps,x_1,\xi_1}^t(\rho)\right)_{t=0}. 
\end{equation}
associated to this new flow. More precisely, one has
\begin{lemm}\label{l:pert-vf} One has, for every $\rho=(x,\xi)$ in $S^*M$,
\begin{equation}\label{e:newflow}
Y_{x_1,\xi_1}^{\eps}(\rho)
=c_{\eps,x_1,\xi_1}(x)X_0(\rho)+\frac{\eps}{\sqrt{2}\|\xi_1\|c_{\eps,x_1,\xi_1}(x)} g_x^*\left(d_xV,\xi^{\perp}\right)\left(X^s(\rho)-X^u(\rho)\right),
\end{equation}
 where $c_{\eps,x_1,\xi_1}(x):=\sqrt{\frac{p_{\eps}(x_1,\xi_1)-\eps V(x)}{p_0(x_1,\xi_1)}}.$

\end{lemm}

The proof of this lemma follows from a direct calculation and it was given in~\cite{EsRi14} (lemma $5.2$).

\begin{rema} 
We can reestablish the dependence in $(x_1,\xi_1)$ more clearly and write:
$$c_{\eps,x_1,\xi_1}(x)=\sqrt{1+\frac{2}{\|\xi_1\|^2}\eps(V(x_1)-V(x))}=1+\ml{O}_{x,x_1,\xi_1}(\eps).$$
\end{rema}
We can rewrite $I_{x_0,\xi_0}(\eps, b, T)$ using this new flow, and we get
\begin{equation}\label{e:int-SM}I_{x_0,\xi_0}(\eps, b, T)=
 \frac{1}{b}\int_0^{b}\tilde{a}\circ \varphi_{\eps,x(s),\xi(s)}^{T\|\xi_0\|}\circ G_0^{s\|\xi_0\|}\left(x_0,\frac{\xi_0}{\|\xi_0\|}\right)ds,
\end{equation}
where $(x(s),\xi(s)):=G_0^s(x_0,\xi_0)$.

\subsection{Applying strong structural stability}\label{ss:stability}

We will now use the strong structural stability theorem in order to transform the integral $I_{x_0,\xi_0}(\eps, b, T)$ into an integral involving the horocyle flow. 
Precisely, we start by proving the following lemma:
\begin{lemm}\label{l:stability}Suppose that $\dim(M)=2$, and that $M$ has constant negative sectional curvature $K\equiv -1$. Let $\tilde{a}$ be a $\ml{C}^1$ function on $T^*M-M$ 
which is $0$-homogeneous, let $N\geq 1$ and let $0<\gamma<1/2$. 

There exist $\delta_1,\eps_1,s_1, T_1>0$ and $C_1>0$ such that, for every $\eps\in[0,\eps_1]$, for every $s\in[0,s_1]$, for every $T\geq T_1$ and for every
$(x_0,\xi_0)$ satisfying $p_0(x_0,\xi_0)\in[1/2-\delta_1,1/2+\delta_1]$, one has
$$\left|
\tilde{a}\circ \varphi_{\eps,x(s),\xi(s)}^{T\|\xi_0\|}\circ G_0^{s\|\xi_0\|}\left(x_0,\frac{\xi_0}{\|\xi_0\|}\right)- \tilde{a}\circ H_u^{-\eps P_{x_0,\xi_0}^N(s) e^{T\|\xi_0\|}
}(\rho(\eps,T,x_0,\xi_0)) \right|$$
$$\hspace{5cm}\leq C_1\left(\eps T+\eps^{1+\gamma}e^{T\|\xi_0\|}+s+s^{N+1}\eps e^{T\|\xi_0\|}\right),$$
where 
$$\rho(\eps,T,x_0,\xi_0):=G_0^{T\|\xi_0\|}\circ H_u^{\eps 
\beta_V^u\left(x_0,\frac{\xi_0}{\|\xi_0\|}\right)}\left(x_0,\frac{\xi_0}{\|\xi_0\|}\right),$$
and 
\begin{equation}\label{e:polynom}P_{x_0,\xi_0}^N(s):=\sum_{p=0}^{N-1}\frac{(\|\xi_0\|s)^{p+1}}{(p+1)!}\left(\left(X_0-1\right)^p.f_V\right)\left(x_0,\frac{\xi_0}{\|\xi_0\|}\right).\end{equation}
\end{lemm}

 In the end, we will need to average over the time parameter $s$. Thanks to this lemma, this average will now correspond to an average along ``reparametrized'' trajectories of the horocycle flow which is known to be uniquely ergodic~\cite{Fu73, Mar77}.

\begin{rema} We emphasize that our assumption on the set $K_{V}^J(\eta_0)$ involves some nonvanishing conditions on the coefficients of the polynom appearing in the ``reparametrization'' of the horocycle flow.
 
\end{rema}

\begin{rema}\label{r:crit-scale} Before giving the proof of this lemma, we start with a simple observation which explains why we need to consider time scales larger than $|\log(\eps_0)|$ in the statement of proposition~\ref{p:dynamics}. In fact, if one has $\eps\ll e^{-T\|\xi_0\|}$ and $s\ll 1$, then the previous lemma implies that
 $$\left|
\tilde{a}\circ \varphi_{\eps,x(s),\xi(s)}^{T\|\xi_0\|}\circ G_0^{s\|\xi_0\|}\left(x_0,\frac{\xi_0}{\|\xi_0\|}\right)- \tilde{a}\circ G_0^{T}\left(x_0,\xi_0\right) \right|=o(1).$$
In particular, averaging over the time parameter $s$ for such scales would not provide any equidistribution.
\end{rema}

\begin{proof} Let $\tilde{a}$ be a smooth function on $T^*M-M$ which is $0$-homogeneous. We will use the conventions of appendix~\ref{a:stability}. We fix $0<\gamma<1/2$. 
Let $(x_0,\xi_0)$ be an element in a small neighborhood of $S^*M$. 

First, we write the strong structural stability equation. More precisely, thanks 
to~\eqref{e:struc-stab-useful}, we have, for every $\rho$ in $S^*M$ and for every $s$ in $\IR$,
$$\tilde{a}\circ \varphi_{\eps,x(s),\xi(s)}^{T\|\xi_0\|}(\rho)=\tilde{a}\circ h^{\eps}_{x(s),\xi(s)}\circ G_0^{\tau_{x(s),\xi(s)}^{\eps}\left(T\|\xi_0\|,\rho\right)}\circ
\left(h^{\eps}_{x(s),\xi(s)}\right)^{-1}(\rho).$$
Using the ``smoothness'' of the maps withr espect to $\eps$ -- see~\eqref{e:smoothmap}, we can write that
$$\tilde{a}\circ \varphi_{\eps,x(s),\xi(s)}^{T\|\xi_0\|}(\rho)=\tilde{a}\circ G_0^{T\|\xi_0\|}\circ
\left(h^{\eps}_{x(s),\xi(s)}\right)^{-1}(\rho)+\ml{O}(\eps T),$$
where the constant in the remainder is uniform for $\rho$ in $S^*M$, $s\in\IR$ and $(x_0,\xi_0)$ in a small neighborhood of $S^*M$.

We now replace $\left(h^{\eps}_{x(s),\xi(s)}\right)^{-1}$ by its approximate expression given by~\eqref{e:inverse-holder}. Then, according to~\eqref{e:exp-rate}, we get
$$\tilde{a}\circ \varphi_{\eps,x(s),\xi(s)}^{T\|\xi_0\|}(\rho)=\tilde{a}\circ G_0^{T\|\xi_0\|}\circ
\exp\left(-\tilde{v}^{\eps}_{x(s),\xi(s)}\right)(\rho)+\ml{O}(\eps T)+\ml{O}(\eps^{1+\gamma}e^{T\|\xi_0\|}),$$
where $\tilde{v}^{\eps}_{x(s),\xi(s)}$ is defined by~\eqref{e:approx-inverse} and is $\ml{C}^1$ in $\eps$. Again, the constants in the remainders are still uniform for $\rho$ in $S^*M$, $s\in\IR$ and $(x_0,\xi_0)$ in a small neighborhood of $S^*M$.

Thanks to lemma~$1$ in~\cite{Mo69} -- see also remark~\ref{r:moser} from the appendix, we have
$$d_{\ml{C}^0}\left(\exp\left(-\tilde{v}^{\eps}_{x(s),\xi(s)}\right),\exp\left(-\frac{\eps}{\|\xi_0\|}\beta_V^sX^s\right)\circ 
\exp\left(-\frac{\eps}{\|\xi_0\|}\beta_V^uX^u\right)\right)=\ml{O}(\eps^{1+\gamma}),$$
where the constant is uniform for $s\in\IR$ and $(x_0,\xi_0)$ in a small neighborhood of $S^*M$. In particular, one has
$$\tilde{a}\circ \varphi_{\eps,x(s),\xi(s)}^{T\|\xi_0\|}(\rho)=\tilde{a}\circ G_0^{T\|\xi_0\|}\circ
\exp\left(-\frac{\eps}{\|\xi_0\|}\beta_V^sX^s\right)\circ 
\exp\left(-\frac{\eps}{\|\xi_0\|}\beta_V^uX^u\right)(\rho)+\ml{O}(\eps T)+\ml{O}(\eps^{1+\gamma}e^{T\|\xi_0\|}).$$
We will now approximate these two maps by the unstable and stable horocycle flows.  For that purpose, we fix $\rho$ in $S^*M$and we observe that the maps
$$\eps\mapsto \exp_{\rho}\left(-\frac{\eps}{\|\xi_0\|}\beta_V^uX^u\right),\ \text{and}\ \eps\mapsto H_u^{-\frac{\eps}{\|\xi_0\|}\beta_V^u(\rho)}(\rho)$$ 
are of class $\ml{C}^1$. Moreover, their derivatives coincides up to an error of order $\ml{O}(\eps)$ where the constant in the remainder is uniform for $\rho$ in $S^*M$ and $\|\xi_0\|$ close to $1$. In particular, 
$$d_{S^*M} \left(\exp_{\rho}\left(-\frac{\eps}{\|\xi_0\|}\beta_V^uX^u\right), H_u^{-\frac{\eps}{\|\xi_0\|}\beta_V^u(\rho)}(\rho)\right)=\ml{O}(\eps^{2}),$$
where the constant in the remainder is still uniform for $\rho$ in $S^*M$ and $\|\xi_0\|$ close to $1$. The same holds for the maps generated by the stable vector field $X^s$. Thus, up to other error terms of the same order, we can replace the maps $ \exp\left(-\frac{\eps}{\|\xi_0\|}\beta_V^sX^s\right)$ and $ \exp\left(-\frac{\eps}{\|\xi_0\|}\beta_V^uX^u\right)$ by the stable and unstable horocycle flows, i.e.
$$\tilde{a}\circ \varphi_{\eps,x(s),\xi(s)}^{T\|\xi_0\|}(\rho)=\tilde{a}\circ G_0^{T\|\xi_0\|}\circ
H_s^{-\frac{\eps}{\|\xi_0\|}\beta_V^s}\circ
H_u^{-\frac{\eps}{\|\xi_0\|}\beta_V^u}(\rho)+\ml{O}(\eps T)+\ml{O}(\eps^{1+\gamma}e^{T\|\xi_0\|}),$$
where the constant in the remainders are still uniform for $\rho$ in $S^*M$, 
$s\in\IR$ and $(x_0,\xi_0)$ in a small neighborhood of $S^*M$. This implies that
$$\tilde{a}\circ \varphi_{\eps,x(s),\xi(s)}^{T\|\xi_0\|}(\rho)=\tilde{a}\circ G_0^{T\|\xi_0\|}\circ
H_u^{-\frac{\eps}{\|\xi_0\|}\beta_V^u(\rho)}(\rho)+\ml{O}(\eps T)+\ml{O}(\eps^{1+\gamma}e^{T\|\xi_0\|}),$$
as $G_0^t\circ H_s^{\tau}=H_s^{e^{-t}\tau}\circ G_0^t$ for every $t$ and $\tau$ in $\IR$~\cite{Mar77}. Up to this point, the proof is exactly the same as in~\cite{EsRi14}, and we will now analyse more precisely the reparametrization constant $\beta^u(\rho)$.

We use the fact that $G_0^t\circ H_u^{\tau}=H_u^{e^{t}\tau}\circ G_0^t$, and we find
\begin{equation}\label{e:horocycle-step1}\tilde{a}\circ \varphi_{\eps,x(s),\xi(s)}^{T\|\xi_0\|}(\rho)=\tilde{a}\circ G_0^{s\|\xi_0\|}\circ
H_u^{-e^{(T-s)\|\xi_0\|}\frac{\eps\beta_V^u(\rho)}{\|\xi_0\|}}\circ G_0^{(T-s)\|\xi_0\|}(\rho)+\ml{O}(\eps T)+\ml{O}(\eps^{1+\gamma}e^{T\|\xi_0\|}).\end{equation}
We now write that
$$e^{-s\|\xi_0\|}\beta_V^u\circ G_0^{s\|\xi_0\|}\left(x_0,\frac{\xi_0}{\|\xi_0\|}\right)=\beta_V^u\left(x_0,\frac{\xi_0}{\|\xi_0\|}\right)
-\int_0^{s\|\xi_0\|}e^{-\tau}f_V\circ G_0^{\tau}\left(x_0,\frac{\xi_0}{\|\xi_0\|}\right)d\tau,$$
where $f_V(x,\xi)=g_x^{*}(d_xV,\xi^{\perp}).$ Then, we find that, for every $p\geq 0$,
$$\frac{d^{p+1}}{ds^{p+1}}\left(e^{-s\|\xi_0\|}\beta_V^u\circ G_0^{s\|\xi_0\|}\left(x_0,\frac{\xi_0}{\|\xi_0\|}\right)\right)_{s=0}
=\|\xi_0\|^{p+1}(X_0-1)^p.f_V\left(x_0,\frac{\xi_0}{\|\xi_0\|}\right).$$
In particular, we can write the order $N$ expansion as $s\rightarrow 0$, i.e.
\begin{equation}\label{e:horocycle-step2}e^{-s\|\xi_0\|}\beta_V^u\circ G_0^{s\|\xi_0\|}\left(x_0,\frac{\xi_0}{\|\xi_0\|}\right)=\beta_V^u\left(x_0,\frac{\xi_0}{\|\xi_0\|}\right)-
P_{x_0,\xi_0}^N(s)+\ml{O}(s^{N+1}),\end{equation}
where $P_{x_0,\xi_0}^N(s)$ is defined by~\eqref{e:polynom}.
Finally, combining~\eqref{e:horocycle-step1} and~\eqref{e:horocycle-step2}, we find that
\begin{align*}
\tilde{a}\circ \varphi_{\eps,x(s),\xi(s)}^{T\|\xi_0\|}\circ G_0^{s\|\xi_0\|}\left(x_0,\frac{\xi_0}{\|\xi_0\|}\right) & = \tilde{a}\circ H_u^{\eps e^{T\|\xi_0\|}\left(
\beta_V^u\left(x_0,\frac{\xi_0}{\|\xi_0\|}\right)-P_{x_0,\xi_0}^N(s)\right)}\circ G_0^{T\|\xi_0\|}\left(x_0,\frac{\xi_0}{\|\xi_0\|}\right) \\
 & + \ml{O}(\eps T)+\ml{O}(\eps^{1+\gamma}e^{T\|\xi_0\|})+\ml{O}(s)+\ml{O}(s^{N+1}\eps e^{T\|\xi_0\|}),
\end{align*}
where the constant in the remainders are uniform for $(x_0,\xi_0)$ in a small neighborhood of $S^*M$. This concludes the proof of the lemma.

\end{proof}

\subsection{Using unique ergodicity of the horocycle flow}\label{ss:unique-erg}
Thanks to lemma~\ref{l:stability}, we can write, for every $N\geq 1$
\begin{equation}\label{e:integral-remainder}
I_{x_0,\xi_0}(\eps, b, T) =
 \frac{1}{b}\int_0^{b}\tilde{a}\circ \circ H_u^{-\eps P_{x_0,\xi_0}^N(s) e^{T\|\xi_0\|}
}(\rho(\eps,T,x_0,\xi_0))ds
 + \ml{O}(\eps T)+\ml{O}(\eps^{1+\gamma}e^{T\|\xi_0\|})+\ml{O}(b)+\ml{O}(b^{N+1}\eps e^{T\|\xi_0\|}).
\end{equation}
Regarding the previous formula, we need to understand the asymptotic behaviour of
$$\tilde{I}_{x_0,\xi_0}(\eps, b, T) :=
 \frac{1}{b}\int_0^{b}\tilde{a}\circ \circ H_u^{-\eps P_{x_0,\xi_0}^N(s) e^{T\|\xi_0\|}
}(\rho(\eps,T,x_0,\xi_0))ds$$
as $\eps, b\rightarrow 0$, and $T\rightarrow +\infty.$ We are now in a situation which looks very much like~\eqref{e:mixing}. The main difference is the polynomial term in the time reparametrization.

It would be natural to make the change of variables $s'=P_{x_0,\xi_0}^N(s)$. However, the polynom $(P_{x_0,\xi_0}^N)'(s)$ may vanish on certain points of the interval. 
This is the reason why we require the point $(x_0,\xi_0)$ to belong to the subset $K_V^J(\eta_0)$. This hypothesis means that, for such points, at least one the first $J+1$ coefficients of 
the polynom $(P_{x_0,\xi_0}^N)'(s)$ does not vanish. In particular, a first observation we can make is that the Jacobian factor $(P_{x_0,\xi_0}^N)'(s)$ can only vanish at 
finitely many places.

Our first step will be to understand precisely the subsets where the Jacobian of the change of variables is very small (paragraph~\ref{sss:jacobian}). 
Then, it will allow us to make the change of variables on proper subintervals of $[0,b]$ (paragraph~\ref{sss:ch-var}) and to apply unique ergodicity of the horocycle flow (paragraph~\ref{sss:uni-erg}).

\subsubsection{Estimates on the Jacobian of the change of variables}\label{sss:jacobian}
In order to study the size of the Jacobian in our change of variables, we will proceed as in~\cite{EsRi14}, i.e. make use of the following theorem due to Cartan~\cite{Ca28}:
\begin{theo} Given any number $H>0$ and any complex numbers $z_1, \ldots, z_n$, there is a system of $p$ disks in the complex plane, with $p\leq n$ and with the sum of 
the radii equal to $2H$, such that for each point $z$ lying outside these disks, one has the inequality  
$$|z-z_1|.|z-z_2|.\ldots.|z-z_n|>\left(\frac{H}{e}\right)^n.$$
\end{theo}
From this theorem, one can in fact deduce the following property
\begin{lemm}\label{l:jacob} Let $\theta$ be some positive parameter satisfying $0<\theta<1$. Let $\eta_0>0$, $J\geq 0$ and $N\geq J+1$.

Then, there exists $\delta_1>0$, $0<s_1<1$ and $C_0>0$ such that, for every $0<s_0<s_1$, and for every $(x_0,\xi_0)$ in $K_V^J(\eta_0)$ satisfying 
$1/2-\delta_1\leq p_0(x_0,\xi_0)\leq 1/2+\delta_1$, one has a system of subintervals, some of which can be empty, $L_1(x_0,\xi_0),\ldots, L_J(x_0,\xi_0)$ of $[0,s_0]$ with 
the sum of their length bounded by $C_0 s_0^{1+\frac{\theta}{2(J+1)}}$ and satisfying
$$A_{x_0,\xi_0}(s_0):=\left\{s\in[0,s_0]:|(P_{x_0,\xi_0}^N)'(s)|\leq s_0^{J+\theta}\right\}\subset\bigcup_{p=1}^JL_p(x_0,\xi_0).$$
\end{lemm}

\begin{proof} Modulo minor modifications, the proof follows the same lines as the proof of proposition $5.17$ in~\cite{EsRi14}. For the sake of completeness, we briefly recall how one can deduce this lemma from Cartan's Theorem. We fix $\theta>0$. First, we write
 $$(P_{x_0,\xi_0}^N)'(s)=\sum_{p=0}^{N-1}\left(\frac{\|\xi_0\|^{p+1}}{p!}\left(X_0-1\right)^p.f_V\left(x_0,\frac{\xi_0}{\|\xi_0\|}\right)\right)s^p.$$
If we choose $\delta_1>0$ small enough, one knows that, for every $(x_0,\xi_0)$ in $K_V^J(\eta_0)$ satisfying 
$1/2-\delta_1\leq p_0(x_0,\xi_0)\leq 1/2+\delta_1$, there exists $0\leq p_1\leq J$ such that
$$\left(\frac{\|\xi_0\|^{p_1+1}}{p_1!}\left(X_0-1\right)^{p_1}.f_V\left(x_0,\frac{\xi_0}{\|\xi_0\|}\right)\right)\geq \frac{\eta_0}{2J!}.$$ 

Let $(x_0,\xi_0)$ be such a point and let $p_1$ be the first integer for which the coefficient of the polynom is $\geq \frac{\eta_0}{2J!}$. The case $p_1=0$ is straightforward as we can choose all the intervals to be empty. Suppose now $p_1\neq 0$. 

In this case, we introduce
$$q_{x_0,\xi_0}(s):=\sum_{p=0}^{p_1}\left(\frac{\|\xi_0\|^{p+1}}{p!}\left(X_0-1\right)^{p}.f_V\left(x_0,\frac{\xi_0}{\|\xi_0\|}\right)\right)s^p,$$
which will be the ``dominant part'' of the Jacobian factor $(P_{x_0,\xi_0}^N)'(s)$. In fact, one can define 
$$B_{x_0,\xi_0}(s_0):=\left\{s\in[0,s_0]: |(P_{x_0,\xi_0}^N)'(s)|\leq s_0^{J+\theta},\ \text{and}\ |q_{x_0,\xi_0}(s)|\geq s_0^{p_1+\frac{\theta}{2}}\right\}.$$
Then, there exists an uniform constant $C_{J,N,\delta_1}>0$ such that the following holds
$$B_{x_0,\xi_0}(s_0)\subset\left\{s\in[0,s_0]: C_{J,N,\delta_1}s^{p_1+1}\geq s_0^{p_1+\frac{\theta}{2}}(1-s_0^{J-p_1+\frac{\theta}{2}})\right\}.$$
As $J-p_0+\frac{\theta}{2}>0$ and as $\theta<1$, this set is empty for $s_0>0$ small enough (depending only on $C_{J,N,\delta_1}$, $J$, $\theta$ and $\theta'$). This shows that the dominant part of the Jacobian is encoded by the polynom $q_{x_0,\xi_0}(s)$. In other words, for $s_0>0$ 
small enough, one has
$$A_{x_0,\xi_0}(s_0)\subset\left\{s\in[0,s_0]: \frac{p_1!|q_{x_0,\xi_0}(s)|}{\left|(\|\xi_0\|)^{p_1+1}\left(X_0-1\right)^{p_1}.f_V\left(x_0,\frac{\xi_0}{\|\xi_0\|}\right)\right|}
\leq \frac{2 J!s_0^{p_1+\frac{\theta}{2}}}{\eta_0}\right\}.$$
We are now in a situation where we can apply Cartan's Theorem on polynoms. Thus, there exists a system of subintervals, some of which can be empty, 
$L_1(x_0,\xi_0),\ldots, L_J(x_0,\xi_0)$ of $[0,s_0]$ with the sum of their length bounded by $\frac{2J!e}{\eta_0} s_0^{1+\frac{\theta}{2p_1}}$ and satisfying
$$A_{x_0,\xi_0}(s_0)\subset\bigcup_{p=1}^JL_p(x_0,\xi_0),$$
which concludes the proof of the lemma.
\end{proof}

\subsubsection{Change of variables}\label{sss:ch-var} We will now perform a change of variables in the integral defining $\tilde{I}_{x_0,\xi_0}(\eps, b, T)$. For that purpose, we will split the interval $[0,b]$ in small subintervals where the Jacobian of the change of variables is large enough. Let $0<\theta<1$.

Let $(x_0,\xi_0)$ be an element in $K_V^J(\eta_0)$ satisfying $1/2-\delta_1\leq p_0(x_0,\xi_0)\leq 1/2+\delta_1$, where $\delta_1$ is given by lemma~\ref{l:jacob}. For $b>0$ small enough, the subset $A_{x_0,\xi_0}(b)\cap [0,b]$ is included in the union of at most $J$ subintervals of $[0,b]$ whose total length is bounded by $Cb^{1+\frac{\theta}{2(J+1)}}$. Outside of these ``bad'' subintervals, the Jacobian of the change of variables does not vanish and is bounded from below by $b^{J+\theta}$. We will now split the complementary of these ``bad'' subintervals into a family of $L$ disjoint 
subintervals $(B_k)_{k=1,\ldots L}$ (depending on $x_0,\xi_0$) of individual length $b^{1+2J+2\theta}$ and the union of at most $J+1$ intervals whose total length is bounded by $(J+1)b^{1+2J+2\theta}$. More precisely, we write
\begin{equation}\label{e:splitting-the-integral}\tilde{I}_{x_0,\xi_0}(\eps, b, T)=\frac{1}{b}\sum_{k=1}^L\int_{B_k}
\tilde{a}\circ  H_u^{-\eps P_{x_0,\xi_0}^N(s) e^{T\|\xi_0\|}
}(\rho(\eps,T,x_0,\xi_0))ds+\ml{O}(b^{\frac{\theta}{2(J+1)}}),\end{equation}
where one has
\begin{itemize}
 \item for every $1\leq k'\neq k\leq L$, $B_{k'}\cap B_k=\emptyset$;
 \item for every $1\leq k\leq L$, $B_k$ is an interval of length $b^{1+2J+2\theta}$;
 \item for every $1\leq k\leq L$ and for every $s$ in $B_k$, one has $|(P_{x_0,\xi_0}^N)'(s)|\geq b^{J+\theta}$;
 \item $\frac{1}{b}\sum_{k=1}^L|B_k|=1+\ml{O}(b^{\frac{\theta}{2(J+1)}}).$
\end{itemize}
We will now consider each of the subintegrals independently, i.e. for every $1\leq k\leq L$
 $$\tilde{I}_{x_0,\xi_0}^k(\eps, b, T) :=
 \frac{1}{|B_k|}\int_{B_k}\tilde{a} \circ H_u^{-\eps P_{x_0,\xi_0}^N(s) e^{T\|\xi_0\|}
}(\rho(\eps,T,x_0,\xi_0))ds,$$
and we will verify that this converges to $\int_{S^*M}\tilde{a}dL$ for a proper range of $\eps,\ b\rightarrow 0$, and $T\rightarrow+\infty.$ We fix $1\leq k\leq L$ and we make the change of variables $s'=P_{x_0,\xi_0}^N(s)$ where $s\in B_k$. We obtain
 $$\tilde{I}_{x_0,\xi_0}^k(\eps, b, T) =
 \frac{1}{|B_k|}\int_{P_{x_0,\xi_0}^N(B_k)}\tilde{a} \circ H_u^{-\eps s' e^{T\|\xi_0\|}
}(\rho(\eps,T,x_0,\xi_0))\frac{ds'}{|(P_{x_0,\xi_0}^N)'\circ (P_{x_0,\xi_0}^N\rceil_{B_k})^{-1}(s')|}.$$

\begin{rema} We observe that $P_{x_0,\xi_0}^N(B_k)$ is an interval whose length is bounded from below by $b^{1+3J+3\theta}$ and from above by $\ml{O}(b^{1+2J+2\theta}).$ In the following, we will denote by $\tilde{B}_k$ this interval.
\end{rema}
We now write
$$\left| \frac{1}{\left|(P_{x_0,\xi_0}^J)'\circ (P_{x_0,\xi_0}^J\rceil_{B_k})^{-1}(s')\right|} - \frac{1}{\left|(P_{x_0,\xi_0}^J)'(s_k))\right|} \right| 
\leq \sup_{s \in B_k} \bigg| \frac{|(P_{x_0,\xi_0}^J)''(s)}{(P_{x_0,\xi_0}^J)'\big( s \big)^2} \bigg| \times |(P_{x_0,\xi_0}^J\rceil_{B_k})^{-1}(s') - s_k |,$$ 
where $s_k$ is the left end point of $B_k$. In particular, we find that
$$\tilde{I}_{x_0,\xi_0}^k(\eps, b, T) =
 \frac{1}{|B_k|\left|(P_{x_0,\xi_0}^J)'(s_k))\right|}\int_{\tilde{B}_k}\tilde{a} \circ H_u^{-\eps s' e^{T\|\xi_0\|}
}(\rho(\eps,T,x_0,\xi_0))ds'+\ml{O}(b),$$
where the constant in the remainder is uniform for $1\leq k\leq L$ (and for $(x_0,\xi_0)$ in the allowed energy layers). 

\begin{rema}
Taking the particular case $\tilde{a}=1$, we also observe that 
$$|\tilde{B}_k|=|B_k|\left|(P_{x_0,\xi_0}^J)'(s_k))\right|(1+\ml{O}(b)).$$
\end{rema}
In the end, we have obtained that
\begin{equation}\label{e:integral-we-can-estimate}
 \tilde{I}_{x_0,\xi_0}^k(\eps, b, T) =
 \frac{1}{|\tilde{B}_k|}\int_{\tilde{B}_k}\tilde{a} \circ H_u^{-\eps s' e^{T\|\xi_0\|}
}(\rho(\eps,T,x_0,\xi_0))ds'+\ml{O}(b),
\end{equation}
where $\tilde{B}_k$ is an interval whose length is bounded from below by $b^{1+3J+3\theta}$.

\begin{rema}\label{r:concl-repara} Let $(x_0,\xi_0)$ be an element in $K_V^J(\eta_0)$ satisfying $1/2-\delta_1\leq p_0(x_0,\xi_0)\leq 1/2+\delta_1$, 
where $\delta_1>0$ was given by lemma~\ref{l:jacob}. Combining~\eqref{e:splitting-the-integral} and~\eqref{e:integral-we-can-estimate}, we have that, for every $0<\theta<1$,
$$\tilde{I}_{x_0,\xi_0}(\eps, b, T)=\frac{1}{L}\sum_{k=1}^L\frac{1}{|\tilde{B}_k|}\int_{\tilde{B}_k}\tilde{a} \circ H_u^{-\eps s e^{T\|\xi_0\|}
}(\rho(\eps,T,x_0,\xi_0))ds+\ml{O}\left(b^{\frac{\theta}{2(J+1)}}\right),$$
where $\tilde{B}_k$ is an interval whose length is bounded from below by $b^{1+3J+3\theta}$ and which depends on $(x_0,\xi_0)$ (but not on $\eps$ and $T$).
\end{rema}

\subsubsection{Unique ergodicity}\label{sss:uni-erg}

We can now conclude using unique ergodicity of the horocycle flow~\cite{Fu73, Mar77}, which implies 
$$r(T_0):=\sup_{|\tau|\geq T_0}\sup\left\{\left|\frac{1}{\tau}\int_0^{\tau}\tilde{a}\circ H_u^{s'}(\rho)ds'-\int_{S^*M}\tilde{a}dL\right|:\rho\in S^*M\right\}$$
is a nonincreasing function which tends to $0$ as $T_0\rightarrow +\infty$. Then, one has, for every $1\leq k\leq L$,
$$\frac{1}{|\tilde{B}_k|}\int_{\tilde{B}_k}\tilde{a} \circ H_u^{-\eps s e^{T\|\xi_0\|}
}(\rho(\eps,T,x_0,\xi_0))ds=\int_{S^*M}\tilde{a}dL+r\left(b^{1+3J+3\theta}\eps e^{T\|\xi_0\|}\right).$$
Thanks to remark~\ref{r:concl-repara} and to~\eqref{e:integral-remainder}, it implies that
$$I_{x_0,\xi_0}(\eps, b, T)=\int_{S^*M}\tilde{a}dL+r\left(b^{1+3J+3\theta}\eps e^{T\|\xi_0\|}\right)+\ml{O}\left(b^{\frac{\theta}{2(J+1)}}\right)
+ \ml{O}(\eps T)+\ml{O}(\eps^{1+\gamma}e^{T\|\xi_0\|})+\ml{O}(b^{N+1}\eps e^{T\|\xi_0\|}).$$
As $N$ can be chosen arbitrarly large, $\theta>0$ arbitrarly small and $\gamma$ arbitrarly close to $1/2$, this concludes the proof of proposition~\ref{p:dynamics}, where we took $\eps_0\rightarrow 0^+$ with $b=\eps_0^{\nu_2}$, $T=c|\log\eps_0|$, and $\eps\in[\eps_0^{1+\nu_1},\eps_0]$.

\section{Decay of the quantum Loschmidt echo}
\label{s:loschmidt}

Motivated by the fact that the unitarity of the quantum propagator $e^{-\frac{i\tau \hat{P}_0(\hbar)}{\hbar}}$ (with $\hat{P}_0(\hbar):=-\frac{\hbar^2\Delta_g}{2}$) 
does not allow one to observe any sensitivity to perturbations of the initial conditions, Peres argued in~\cite{Pe84} that both the classical and the quantum system should 
be sensitive to pertubations of the Hamiltonian. For that reason, he suggested that one should look at \emph{perturbations of the Hamiltonian for fixed sequences of initial data} 
in order to study the influence of perturbations both in the classical and in the quantum setting. For the quantum counterpart, he proposed to look at the overlap between the solutions of the unperturbed and the perturbed Schr\"odinger equation for fixed initial data. 
Precisely, given a normalized sequence of initial data $(\psi_{\hbar})_{0<\hbar\leq 1}$ and $V\in\ml{C}^{\infty}(M,\IR)$, one should study the following quantity:
\begin{equation}\label{e:loschmidt}\mathbf{F}_{\hbar,\eps}^V(\psi_{\hbar},\tau)
:=\left|\left\la e^{-\frac{i\tau \hat{P}_0(\hbar)}{\hbar}}\psi_{\hbar},e^{-\frac{i\tau (\hat{P}_0(\hbar)+\eps V)}{\hbar}}\psi_{\hbar}\right\ra\right|^2. 
\end{equation}
Peres expected that this overlap should typically decay for any quantum system, and that it should decay to a much lower value for chaotic sytems than for regular ones. One 
of the main difficutly one encounters when studying this overlap is that we want to understand the limit as $\tau\rightarrow+\infty$ but also as $\hbar\rightarrow 0$ and 
$\eps\rightarrow 0$. In~\cite{JalPas01}, motivated by experiments in nuclear magnetic resonance, Jalabert and Pastawski were also interested\footnote{It seems that the terminology 
``quantum Loschmidt echo'' was introduced in this article.} in studying properties of $\mathbf{F}_{\hbar,\eps}^V(\psi_{\hbar},\tau)$ for chaotic systems. They considered 
the situation where the initial data are given by a sequence of coherent states which are microlocalized at some point $(x_0,\xi_0)$ in phase space, and where the potential is 
given by $V(x)=u_1V_1(x)+\ldots +u_JV_J(x)$, where the $(u_i)_{i=1,\ldots J}$ are independent random variables. They observed that, on average and for a certain range of parameters, 
the quantum Loschmidt echo is 
exponentially decaying with a rate which is asymptotically given by the mean of the Lyapunov exponents of the classical system. This regime can be observed for times of order 
the Ehrenfest time, and for strong enough perturbations (meaning that $\eps\rightarrow 0$ is large compared with the mean level spacing of $\hat{P}_0(\hbar)$). This regime is 
known as the \emph{Lyapunov regime}. In~\cite{JaSiBe01}, it was emphasized that the situation becomes slightly more complicated for smaller perturbations, and that 
one can observe other kind of regimes like the so-called Fermi golden rule regime (exponential decay with a rate depending on $\eps$). Besides the 
above works, many progresses have been made recently in the physics literature concerning the asymptotic behaviour of $\mathbf{F}_{\hbar,\eps}^V(\psi_{\hbar},\tau)$, and we 
refer to~\cite{GPSZ06, JaPe09, GJPW12} for detailed reviews on these questions. It is important to note that, when studying this problem, the decay rates depend in a subtle way 
on various quantities like $\eps$, $\tau$ and $\hbar$, but also the shape (or the statistical properties) of the perturbation, and the choice of initial data.

The aim of this last section is to use the tools developed in the previous section  for the study of the asymptotic properties of the quantum Loschmidt echo on surfaces with constant negative curvature. We will look at strong 
perturbations, namely $\eps\gg\sqrt{\hbar}$ and at scales of times $\geq|\log\eps|$. For simplicity of exposition\footnote{As in the statements of 
section~\ref{s:semiclassical}, our arguments could be generalized to deal with slightly more general normalized initial data satisfying~\eqref{e:hosc2} for some small 
enough $\nu_0>0$.}, we will only consider sequences of normalized initial data satisfying~\eqref{e:hosc}, i.e.
$$\lim_{R\rightarrow+\infty}\limsup_{\hbar\rightarrow 0^+}\left\Vert \mathbf{1}_{\left[  1-R\hbar,1+R\hbar\right]
}\left(  -\hbar^{2}\Delta\right)  \psi_{\hbar}-\psi_{\hbar}\right\Vert _{L^{2}\left(  M\right)
}=0,\ \text{and}\ \forall\ 0<\hbar\leq 1,\ \|\psi_{\hbar}\|_{L^2(M)}=1.$$
Our main results on these questions are propositions~\ref{p:loschmidt1} and~\ref{p:loschmidt2} which state that the quantum Loschmidt echo becomes asymptotically strictly 
less than $1$. Compared with the results described above, it does not provide any decay rate but is valid for \emph{any} sequence of initial data (with a proper 
localization in frequencies) and for any $V$ satisfying a certain explicit admissibility condition. In particular, we do not have to average over a family of perturbations. 

\subsection{Preliminary lemma}

As a first step in the study of the properties of the quantum Loschmidt echo, we study the restriction of the perturbed propagator $e^{-\frac{it}{\hbar}(\hat{P}_0(\hbar)+\eps_{\hbar}V)}$ on the eigenspaces of the unperturbed Schr\"odinger operator $\hat{P}_0(\hbar)$. The following lemma is the key result of this section:

\begin{lemm}\label{l:loschmidt} Suppose that $\dim(M)=2$, and that $M$ has constant negative sectional curvature $K\equiv -1$. Suppose that $\lim_{\hbar\rightarrow 0}\eps_{\hbar}=0$, and that there exists $0<\nu<\frac{1}{2}$ such that, for $\hbar>0$ small enough, one has
$$\eps_{\hbar}\geq \hbar^{\nu}.$$
Let $J$ be a nonnegative integer such that
$$\ml{C}_V^J:=\bigcap_{j=0}^{J}\left\{(x_0,\xi_0)\in S^*M:(X_0^j.f_V)(x_0,\xi_0)=0\right\}=\emptyset.$$
Let $c_1,c_2,\nu_0>0$ satisfying $1+(3J+1)\nu_0<c_1\leq c_2<\min\left\{3/2,1/(2\nu)\right\}.$

Then, there exists $0<c_0<1$ such that, for every sequence $(\tau_{\hbar})_{0<\hbar\leq 1}$ satisfying
$$c_1|\log(\eps_{\hbar})|\leq\tau_{\hbar}\leq c_2|\log(\eps_{\hbar})|,$$
one has
$$\limsup_{\hbar\rightarrow 0^+}\left\|\Pi(1,\hbar\eps_{\hbar}^{-\nu_0})
e^{-\frac{i\tau_{\hbar}}{\hbar}(\hat{P}_0(\hbar)+\eps_{\hbar}V)}
\Pi(1,\hbar\eps_{\hbar}^{-\nu_0})
\right\|_{L^2(M)\rightarrow L^2(M)}\leq c_0,$$
where
$$\Pi(1,\hbar\eps_{\hbar}^{-\nu_0}):=\mathbf{1}_{[1-\eps_{\hbar}^{-\nu_0}\hbar,1+\eps_{\hbar}^{-\nu_0}\hbar]}(-\hbar^2\Delta_g).$$
\end{lemm}

This lemma shows that, under some geometric assumptions on the perturbation, the norm of the perturbed propogator restricted to the eigenspaces of $\hat{P}_0(\hbar)$ is uniformly strictly less than $1$. In appendix~\ref{a:example}, it is shown that the assumption on $V$ is ``generic''.

\begin{proof} We proceed by contradiction, i.e. we suppose that, for every integer $n\geq 1$, one can find a sequence
$(\tau_{\hbar}^n)_{0<\hbar\leq 1}$ satisfying
$$c_1|\log(\eps_{\hbar})|\leq\tau_{\hbar}^n\leq c_2|\log(\eps_{\hbar})|,$$
and such that
$$\limsup_{\hbar\rightarrow 0^+}\left\|\Pi(1,\hbar\eps_{\hbar}^{-\nu_0})
e^{-\frac{i\tau_{\hbar}^n}{\hbar}(\hat{P}_0(\hbar)+\eps_{\hbar}V)}
\Pi(1,\hbar\eps_{\hbar}^{-\nu_0})
\right\|_{L^2(M)\rightarrow L^2(M)}\geq 1-\frac{1}{n}.$$
Thus, for any integer $n\geq 1$, one can find\footnote{Without loss of generality, we can suppose $0<\hbar_{n+1}<\hbar_n$.} $0<\hbar_n\leq\frac{1}{n}$ and $\psi_{\hbar_n}$ such that $\|\psi_{\hbar_n}\|_{L^2}=1$, $c_1|\log\eps_{\hbar_n}|\leq \tau_{\hbar_n}^n\leq c_2|\log\eps_{\hbar_n}|$, and
\begin{equation}\label{e:contradiction-lemma}\left\|\Pi(1,\hbar_n\eps_{\hbar_n}^{-\nu_0})
e^{-\frac{i\tau_{\hbar_n}^n}{\hbar_n}(\hat{P}_0(\hbar_n)+\eps_{\hbar_n}V)}
\Pi(1,\hbar_n\eps_{\hbar_n}^{-\nu_0})\psi_{\hbar_n}
\right\|_{L^2(M)\rightarrow L^2(M)}\geq 1-\frac{2}{n}.\end{equation}
We define $\tilde{\psi}_{\hbar_n}:=\Pi(1,\hbar_n\eps_{\hbar_n}^{-\nu_0})\psi_{\hbar_n}$, which satisfies, thanks to the previous inequality, $\lim_{n\rightarrow +\infty}\|\tilde{\psi}_{\hbar_n}\|=1$. We will now use two different procedures to compute the limit of the following quantity:
$$A_n:=\int_M(V^*)^2\left|e^{-\frac{i\tau_{\hbar_n}^n}{\hbar_n}(\hat{P}_0(\hbar_n)+\eps_{\hbar_n}V)}\tilde{\psi}_{\hbar_n}
\right|^2dvol_g,$$
where $V^*=V-\int_{S^*M} V\circ \pi dL$ (with $\pi(x,\xi)=x$), and $vol_g$ is the Riemannian volume on $M$. Using proposition~\ref{p:semicl} and the fact that $\ml{C}_V^J$ is empty, we first obtain that
$$\lim_{n\rightarrow+\infty}A_n=\int_{S^*M}\left(V\circ\pi-\int_{S^*M} V\circ \pi dL\right)^2dL.$$

We will now compute the limit of $A_n$ in a slightly different manner, and then get the contradiction. For that purpose, we will admit that the following property holds:
\begin{equation}\label{e:lo-step2}Ve^{-\frac{i\tau_{\hbar_n}^n}{\hbar_n}(\hat{P}_0(\hbar_n)+\eps_{\hbar_n}V)}\tilde{\psi}_{\hbar_n}=e^{-\frac{i\tau_{\hbar_n}^n}{\hbar_n}(\hat{P}_0(\hbar_n)+\eps_{\hbar_n}V)}V\tilde{\psi}_{\hbar_n}+o_{L^2}(1).
\end{equation}
We postpone the proof of this equality to the end, and we first show how it allows us to conclude. As in paragraph~\ref{sss:proof-semicl}, we introduce a smooth cutoff function $\chi_{\delta}$ 
to microlocalize the symbols near $S^*M$. Combining this to relation~\eqref{e:lo-step2}, we obtain:
$$A_n=\left\la \tilde{\psi}_{\hbar_n}, \Op_{\hbar_n}(V^*\chi_{\delta}\circ p_0)\mathbf{V}_{n}\tilde{\psi}_{\hbar_n}\right\ra+o(1).$$ 
where we set
$$\mathbf{V}_{n}:=e^{\frac{i\tau_{\hbar_n}^n}{\hbar_n}(\hat{P}_0(\hbar_n)+\eps_{\hbar_n}V)}\Op_{\hbar_n}(V^*\chi_{\delta}\circ p_{\eps_{\hbar_n}})e^{-\frac{i\tau_{\hbar_n}^n}{\hbar_n}(\hat{P}_0(\hbar_n)+\eps_{\hbar_n}V)}.$$
As in paragraph~\ref{sss:proof-semicl}, we can use the invariance of the state $\tilde{\psi}_{\hbar_n}$ on short intervals of time, precisely~\eqref{e:ev-qm}. We find that
$$A_n=\frac{1}{\eps_{\hbar_n}^{\nu_2}}\int_0^{\eps_{\hbar_n}^{\nu_2}}\left\la \tilde{\psi}_{\hbar_n}, e^{-\frac{is\hbar_n\Delta_g}{2}}\Op_{\hbar_n}(V^*\chi_{\delta}\circ p_0)e^{\frac{is\hbar_n\Delta_g}{2}}e^{-\frac{is\hbar_n\Delta_g}{2}}\mathbf{V}_{n}e^{\frac{is\hbar_n\Delta_g}{2}}\tilde{\psi}_{\hbar_n}\right\ra ds+o(1),$$
for some fixed $\nu_2>\nu_0$. Using Egorov and Calder\'on-Vaillancourt theorems~\cite{Zw12} (Chapters $4$ and $11$), we have  
$$e^{-\frac{is\hbar_n\Delta_g}{2}}\Op_{\hbar_n}(V^*\chi_{\delta}\circ p_0)e^{\frac{is\hbar_n\Delta_g}{2}}=\Op_{\hbar_n}(V^*\chi_{\delta}\circ p_0)+o_{L^2\rightarrow L^2}(1),$$ 
uniformly for $s\in[0,\eps_{\hbar_n}^{\nu_2}]$, and thus 
$$A_n=\left\la \tilde{\psi}_{\hbar_n}, \Op_{\hbar_n}(V^*\chi_{\delta}\circ p_0)\left(\frac{1}{\eps_{\hbar_n}^{\nu_2}}\int_0^{\eps_{\hbar_n}^{\nu_2}}e^{-\frac{is\hbar_n\Delta_g}{2}}\mathbf{V}_{n}e^{\frac{is\hbar_n\Delta_g}{2}}ds\right)\tilde{\psi}_{\hbar_n}\right\ra +o(1).$$
So, it remains to analyse the operator $\frac{1}{\eps_{\hbar_n}^{\nu_2}}\int_0^{\eps_{\hbar_n}^{\nu_2}}e^{-\frac{is\hbar_n\Delta_g}{2}}\mathbf{V}_{n}e^{\frac{is\hbar_n\Delta_g}{2}}ds$ whose complete expression is
$$\frac{1}{\eps_{\hbar_n}^{\nu_2}}\int_0^{\eps_{\hbar_n}^{\nu_2}}e^{-\frac{is\hbar_n\Delta_g}{2}}e^{\frac{i\tau_{\hbar_n}^n}{\hbar_n}(\hat{P}_0(\hbar_n)+\eps_{\hbar_n}V)}\Op_{\hbar_n}(V^*\chi_{\delta}\circ p_{\eps_{\hbar_n}})e^{-\frac{i\tau_{\hbar_n}^n}{\hbar_n}(\hat{P}_0(\hbar_n)+\eps_{\hbar_n}V)}e^{\frac{is\hbar_n\Delta_g}{2}}ds.$$
The proof of proposition~\ref{p:semicl} was in fact reduced to studying the convergence of this kind of operator -- see equations~\eqref{e:step1} and~\eqref{e:step2}. In particular, we proved that, modulo small error terms, this operator is a $\hbar_n$-pseudodifferential operator with principal symbol
$$(x,\xi)\mapsto \chi_{\delta}\circ p_{0}(x,\xi)\frac{1}{\eps_{\hbar_n}^{\nu_2}}\int_0^{\eps_{\hbar_n}^{\nu_2}}V^*\circ G_{\eps_{\hbar_n}}^{\tau_{\hbar_n}}\circ G_0^s(x,\xi)ds.$$
Combining the facts that $\ml{C}_V^J$ is empty and that $\int_{S^*M}V^*dL=0$ to proposition~\ref{p:dynamics}, it can be shown that the norm of this operator is in fact $o(1)$ as $n\rightarrow +\infty$. Thus, one has
$$\int_{S^*M}\left(V\circ\pi-\int_{S^*M} V\circ \pi dL\right)^2dL=\lim_{n\rightarrow +\infty} A_n=0,$$
which provides the contradiction as $\ml{C}_V^J$ is empty.

It remains now to verify that~\eqref{e:lo-step2} holds. We observe that, up to this point, we did not use all the informations contained in~\eqref{e:contradiction-lemma}. In particular, 
by construction of $\tilde{\psi}_{\hbar_n}$, one knows that, as $n\rightarrow+\infty$, 
$$r_n:=\left\|\left(\text{Id}_{L^2}-\Pi(1,\hbar_n\eps_{\hbar_n}^{-\nu_0})\right)e^{-\frac{i\tau_{\hbar_n}^n}{\hbar_n}(\hat{P}_0(\hbar_n)+\eps_{\hbar_n}V)}\tilde{\psi}_{\hbar_n}\right\|_{L^2}=o(1).$$
Then, we fix a bounded sequence $(\delta_n)_{n\geq 1}$ such that $\eps_{\hbar_n}\delta_n^{-1}\rightarrow 0$, and $\delta_n\eps_{\hbar_n}^{-1}r_n\rightarrow 0$ as $n\rightarrow+\infty.$ We let $0\leq \chi_1\leq 1$ 
be a smooth cutoff function which is equal to $1$ in a small neighborhood of $0$ and $0$ outside a slightly larger interval, say $[-1/2,1/2]$. As $1+\nu_0<\frac{1}{\nu}$ by assumption, and as 
$\eps_{\hbar_n}\delta_n^{-1}\rightarrow 0$, we find that $\hbar_n\eps_{\hbar_n}^{-\nu_0}\delta_n^{-1}\rightarrow 0$. As $\tilde{\psi}_{\hbar_n}=\Pi(1,\hbar_n\eps_{\hbar_n}^{-\nu_0})\tilde{\psi}_{\hbar_n}$, this implies that, for $n$ large enough,
\begin{equation}\label{e:cutoff-loschmidt}\tilde{\psi}_{\hbar_n}=\chi_1\left(\frac{\hat{P}_0(\hbar_n)-1/2}{\delta_n}\right)\tilde{\psi}_{\hbar_n}.\end{equation}
Using functional calculus for pseudodifferential operators (Ch.~$14$ in~\cite{Zw12}) and the fact that $\delta_n\geq\hbar^{\nu}$ for some $0<\nu<1/2$, we know that the operators 
$$\chi_1\left(\frac{\hat{P}_0(\hbar_n)-1/2}{\delta_n}\right),\ \text{and}\ \chi_1\left(\frac{\hat{P}_{\eps_{\hbar_n}}(\hbar_n)-1/2}{\delta_n}\right)$$
are $\hbar$-pseudodifferential operators in $\Psi_{\nu}^{-\infty,0}(M)$. Then, using the Calder\'on-Vaillancourt theorem~\cite{Zw12} (Ch.~$5$), we find that
\begin{equation}\label{e:taylor}\left\|\chi_1\left(\frac{\hat{P}_0(\hbar_n)-1/2}{\delta_n}\right)-\chi_1\left(\frac{\hat{P}_{\eps_{\hbar_n}}(\hbar_n)-1/2}{\delta_n}\right)\right\|_{L^2\rightarrow L^2}=\ml{O}(\eps_{\hbar_n}\delta_n^{-1})=o(1).\end{equation}
Using~\eqref{e:cutoff-loschmidt} and twice this equality, we find that
$$Ve^{-\frac{i\tau_{\hbar_n}^n}{\hbar_n}(\hat{P}_0(\hbar_n)+\eps_{\hbar_n}V)}\tilde{\psi}_{\hbar_n}=V\chi_1\left(\frac{\hat{P}_0(\hbar_n)-1/2}{\delta_n}\right)e^{-\frac{i\tau_{\hbar_n}^n}{\hbar_n}(\hat{P}_0(\hbar_n)+\eps_{\hbar_n}V)}\tilde{\psi}_{\hbar_n}+o(1).$$
Thanks to the composition properties of $\hbar$-pseudodifferential operators, this can also be rewritten as
\begin{equation}\label{e:commuting-loschmidt}
Ve^{-\frac{i\tau_{\hbar_n}^n}{\hbar_n}(\hat{P}_0(\hbar_n)+\eps_{\hbar_n}V)}\tilde{\psi}_{\hbar_n}=\chi_1\left(\frac{\hat{P}_0(\hbar_n)-1/2}{\delta_n}\right)Ve^{-\frac{i\tau_{\hbar_n}^n}{\hbar_n}(\hat{P}_0(\hbar_n)+\eps_{\hbar_n}V)}\tilde{\psi}_{\hbar_n}+o(1).
\end{equation}
We now remark that
$$\left\|\left(\frac{\hat{P}_0(\hbar_n)-1/2}{\eps_{\hbar_n}}\right)\chi_1\left(\frac{\hat{P}_0(\hbar_n)-1/2}{\delta_n}\right)e^{-\frac{i\tau_{\hbar_n}^n}{\hbar_n}(\hat{P}_0(\hbar_n)+\eps_{\hbar_n}V)}\tilde{\psi}_{\hbar_n}\right\|_{L^2}=\ml{O}(\hbar_n\eps_{\hbar_n}^{-1-\nu_0})+\ml{O}(r_n\delta_n\eps_{\hbar_n}^{-1}).$$
Implementing this property in~\eqref{e:commuting-loschmidt}, we get
$$Ve^{-\frac{i\tau_{\hbar_n}^n}{\hbar_n}(\hat{P}_0(\hbar_n)+\eps_{\hbar_n}V)}\tilde{\psi}_{\hbar_n}=\chi_1\left(\frac{\hat{P}_0(\hbar_n)-1/2}{\delta_n}\right)e^{-\frac{i\tau_{\hbar_n}^n}{\hbar_n}(\hat{P}_0(\hbar_n)+\eps_{\hbar_n}V)}\frac{\hat{P}_{\eps_{\hbar_n}}(\hbar_n)}{\eps_{\hbar_n}}\tilde{\psi}_{\hbar_n}+o(1),$$
which implies
$$Ve^{-\frac{i\tau_{\hbar_n}^n}{\hbar_n}(\hat{P}_0(\hbar_n)+\eps_{\hbar_n}V)}\tilde{\psi}_{\hbar_n}=\chi_1\left(\frac{\hat{P}_0(\hbar_n)-1/2}{\delta_n}\right)e^{-\frac{i\tau_{\hbar_n}^n}{\hbar_n}(\hat{P}_0(\hbar_n)+\eps_{\hbar_n}V)}V\tilde{\psi}_{\hbar_n}+o(1),$$
Applying~\eqref{e:cutoff-loschmidt},~\eqref{e:taylor} and the composition rule for pseudodifferential operators in the other way, we finally obtain
$$Ve^{-\frac{i\tau_{\hbar_n}^n}{\hbar_n}(\hat{P}_0(\hbar_n)+\eps_{\hbar_n}V)}\tilde{\psi}_{\hbar_n}=e^{-\frac{i\tau_{\hbar_n}^n}{\hbar_n}(\hat{P}_0(\hbar_n)+\eps_{\hbar_n}V)}V\tilde{\psi}_{\hbar_n}+o(1),$$
which is exactly equality~\eqref{e:lo-step2}.
\end{proof}

\subsection{Properties of the quantum Loschmidt echo}

\subsubsection{Times of order $|\log\eps_{\hbar}|$}

As a direct application of lemma~\ref{l:loschmidt}, we obtain the following property of the quantum Loschmidt echo:
\begin{prop}\label{p:loschmidt1} Suppose that $\dim(M)=2$, and that $M$ has constant negative sectional curvature $K\equiv -1$. Suppose also that 
$$\ml{C}_V=\bigcap_{j=0}^{+\infty}\left\{(x_0,\xi_0)\in S^*M:(X_0^j.f_V)(x_0,\xi_0)=0\right\}=\emptyset.$$

Let $(\eps_{\hbar})_{0<\hbar\leq 1}$ be a sequence such that $\lim_{\hbar\rightarrow 0}\eps_{\hbar}=0$, and such that there exists $0<\nu<\frac{1}{2}$ verifying, for $\hbar>0$ small enough,
$$\eps_{\hbar}\geq \hbar^{\nu}.$$
Let $1<c_1\leq c_2<\min\{3/2,1/(2\nu)\}$.

Then, there exists $0\leq c_0<1$ such that, for every sequence $(\psi_{\hbar})_{0<\hbar\leq 1}$ satisfying~\eqref{e:hosc}, and for every sequence $(\tau_{\hbar})_{0<\hbar\leq 1}$ satisfying for $\hbar>0$ small enough
$$c_1|\log\eps_{\hbar}|\leq\tau_{\hbar}\leq c_2|\log\eps_{\hbar}|,$$
one has
$$0\leq \limsup_{\hbar\rightarrow 0}\mathbf{F}_{\hbar,\eps_{\hbar}}^V(\psi_{\hbar},\tau_{\hbar})\leq c_0<1.$$
\end{prop}
 
\begin{rema}\label{r:energy-localization} In order to deduce this proposition from lemma~\ref{l:loschmidt}, one should observe that  every sequence $(\psi_{\hbar})_{0<\hbar\leq 1}$ satisfying~\eqref{e:hosc} verify, for every $\nu_0>0$,
$$\lim_{\hbar\rightarrow 0}\left\|\left(\text{Id}-\mathbf{1}_{[1-\eps_{\hbar}^{-\nu_0}\hbar,1+\eps_{\hbar}^{-\nu_0}\hbar]}(-\hbar^2\Delta_g)\right)\psi_{\hbar}\right\|_{L^2}=0.$$
Moreover, by a compactness argument, $\ml{C}_V=\emptyset$ implies that there exists $J\geq 0$ such that $\ml{C}_V^J=\emptyset$.
\end{rema}

Our result shows that the quantum Loschmidt echo is strictly less than $1$ for times of order $|\log\eps_{\hbar}|$. Recall that, under our assumptions on $\eps_{\hbar}$, we 
are looking at a scale times for which the standard semiclassical rules apply. In other words, we are below the so-called Ehrenfest time which is $\frac{|\log\hbar|}{2}$ in 
this geometric context. We also observe 
that our result holds for \emph{any} sequence of initial data satisfying proper energy localization, namely~\eqref{e:hosc}. Compared with the results from the physics literature 
mentioned in the introduction of this section, we emphasize that our proposition does not provide a priori a decay rate for the quantum Loschmidt echo. 
 
\subsubsection{Beyond the $|\log\eps_{\hbar}|$-scale}

The previous proposition holds for times of order $|\log\eps_{\hbar}|$, and it is natural to ask what are the properties of the quantum Loschmidt echo beyond this time 
scale. According to the physics literature, it should typically be a nonincreasing function of time. Thus, one expects that it will at least remain strictly smaller than $1$ 
for larger times. The following proposition provides some informations in this direction.

\begin{prop}\label{p:loschmidt2} Suppose that $\dim(M)=2$, and that $M$ has constant negative sectional curvature $K\equiv -1$. Suppose also that 
$$\ml{C}_V=\bigcap_{j=0}^{+\infty}\left\{(x_0,\xi_0)\in S^*M:(X_0^j.f_V)(x_0,\xi_0)=0\right\}=\emptyset.$$
Let $(\eps_{\hbar})_{0<\hbar\leq 1}$ be a sequence such that $\lim_{\hbar\rightarrow 0}\eps_{\hbar}=0$, and such that there exists $0<\nu<\frac{1}{2}$ verifying, for $\hbar>0$ small enough,
$$\eps_{\hbar}\geq \hbar^{\nu}.$$

Then, there exists $0\leq c_0<1$ such that, for every sequence $(\psi_{\hbar})_{0<\hbar\leq 1}$ satisfying~\eqref{e:hosc}, and for every sequence $(\tau_{\hbar})_{0<\hbar\leq 1}$ satisfying
$$\lim_{\hbar\rightarrow 0^+}\frac{\tau_{\hbar}}{|\log\eps_{\hbar}|}=+\infty,$$
one has
$$0\leq \limsup_{\hbar\rightarrow 0}\int_0^1\mathbf{F}_{\hbar,\eps_{\hbar}}^V(\psi_{\hbar},t\tau_{\hbar})dt\leq c_0<1.$$

\end{prop}

Compared with proposition~\ref{p:loschmidt1} which was for instance valid \emph{for any time} of order $c|\log\eps_{\hbar}|$ (with $c>0$ fixed in a convenient interval), this 
result holds \emph{on average over an interval of times} of order $\tau_{\hbar}$. The advantage is that we can consider much larger times, e.g. we can choose 
$\tau_{\hbar}\gg|\log\hbar|$. A natural scale of times is $\tau_{\hbar}=1/\hbar$, which is called the Heisenberg time in the physics literature. In this case, the result reads as follows:
\begin{coro} Suppose that $\dim(M)=2$, that $M$ has constant negative sectional curvature $K\equiv -1$, and that 
$\ml{C}_V=\emptyset.$ Let $3/2<\alpha<2$. 

Then, there exists $0\leq c_0<1$ such that for every sequence $(\psi_{\hbar})_{0<\hbar\leq 1}$ satisfying~\eqref{e:hosc}, one has
$$\forall\tau_0>0,\ \ \limsup_{\hbar\rightarrow 0}\frac{1}{\tau_0}\int_0^{\tau_0}\left|\left\la e^{\frac{it\Delta_g}{2}}\psi_{\hbar},e^{\frac{it(\Delta_g+\hbar^{-\alpha}V)}{2}}\psi_{\hbar}\right\ra\right|^2dt\leq c_0<1.$$
\end{coro}

We will now give the proof of proposition~\ref{p:loschmidt2} which also follows from lemma~\ref{l:loschmidt}.

\begin{proof} As $\ml{C}_V=\emptyset$ and as $S^*M$ is compact, there exists $J\geq 0$ such that $\ml{C}_V^J=\emptyset$. We fix $\nu_0>0$ small enough to ensure $1+(3J+1)\nu_0<\min\{3/2,1/(2\nu)\}$ as in the statement of lemma~\ref{l:loschmidt}.

Using remark~\ref{r:energy-localization}, we first observe that
$$\limsup_{\hbar\rightarrow 0}\int_0^1\mathbf{F}_{\hbar,\eps_{\hbar}}^V(\psi_{\hbar},t\tau_{\hbar})dt
\leq\limsup_{\hbar\rightarrow 0}\int_0^1\left\|\Pi(1,\hbar\eps_{\hbar}^{-\nu_0})e^{-\frac{it\tau_{\hbar}}{\hbar}(\hat{P}_0(\hbar)+\eps_{\hbar}V)}\psi_{\hbar}\right\|^2dt.$$
We will in fact show something slightly stronger. Precisely, we will prove that there exists $0\leq c_0<1$ such that for every sequence $(\psi_{\hbar})_{0<\hbar\leq 1}$ satisfying~\eqref{e:hosc}, and for every sequence $(\tau_{\hbar})_{0<\hbar\leq 1}$ satisfying
$$\lim_{\hbar\rightarrow 0^+}\frac{\tau_{\hbar}}{|\log\eps_{\hbar}|}=+\infty,$$
one has
$$0\leq \limsup_{\hbar\rightarrow 0}\int_0^1\left\|\Pi(1,\hbar\eps_{\hbar}^{-\nu_0})e^{-\frac{it\tau_{\hbar}}{\hbar}(\hat{P}_0(\hbar)+\eps_{\hbar}V)}\psi_{\hbar}\right\|^2dt\leq c_0<1.$$
We proceed by contradiction. We suppose that there exist a sequence $(\psi_{\hbar_n})_{n\geq 1}$ satisfying~\eqref{e:hosc}, and a sequence $(\tau_{\hbar_n})_{n\geq 1}$ satisfying
$$\lim_{n\rightarrow +\infty}\frac{\tau_{\hbar_n}}{|\log\eps_{\hbar_n}|}=+\infty,$$
such that
\begin{equation}\label{e:step0-long-time}\lim_{n\rightarrow +\infty}\int_0^1\left\|\Pi(1,\hbar_n\eps_{\hbar_n}^{-\nu_0})e^{-\frac{it\tau_{\hbar_n}}{\hbar_n}(\hat{P}_0(\hbar_n)+\eps_{\hbar_n}V)}\psi_{\hbar_n}\right\|^2dt=1.
\end{equation}
In particular, as $\|\psi_{\hbar_n}\|=1$, one has
\begin{equation}\label{e:step1-long-time}
\lim_{n\rightarrow +\infty}\int_0^1\left\|\left(\text{Id}-\Pi(1,\hbar_n\eps_{\hbar_n}^{-\nu_0})\right)e^{-\frac{it\tau_{\hbar_n}}{\hbar_n}(\hat{P}_0(\hbar_n)+\eps_{\hbar_n}V)}\psi_{\hbar_n}\right\|^2dt=0,
\end{equation}
and also, from the Jensen's inequality,
\begin{equation}\label{e:step2-long-time}
\lim_{n\rightarrow +\infty}\int_0^1\left\|\left(\text{Id}-\Pi(1,\hbar_n\eps_{\hbar_n}^{-\nu_0})\right)e^{-\frac{it\tau_{\hbar_n}}{\hbar_n}(\hat{P}_0(\hbar_n)+\eps_{\hbar_n}V)}\psi_{\hbar_n}\right\|dt=0,
\end{equation}
We fix  $1+(3J+1)\nu_0<c<\min\{3/2,1/(2\nu)\}$. By changing the variables in equation~\eqref{e:step0-long-time}, we deduce that
$$\lim_{n\rightarrow +\infty}\int_0^1\left\|\Pi(1,\hbar_n\eps_{\hbar_n}^{-\nu_0})e^{-\frac{ic|\log\eps_{\hbar_n}|}{\hbar_n}(\hat{P}_0(\hbar_n)+\eps_{\hbar_n}V)}e^{-\frac{it\tau_{\hbar_n}}{\hbar_n}(\hat{P}_0(\hbar_n)+\eps_{\hbar_n}V)}\psi_{\hbar_n}\right\|^2dt=1.$$
Then, from~\eqref{e:step0-long-time}, ~\eqref{e:step1-long-time} and~\eqref{e:step2-long-time}, one gets
$$1=\lim_{n\rightarrow +\infty}\int_0^1\left\|\Pi(1,\hbar_n\eps_{\hbar_n}^{-\nu_0})e^{-\frac{ic|\log\eps_{\hbar_n}|}{\hbar_n}(\hat{P}_0(\hbar_n)+\eps_{\hbar_n}V)}\Pi(1,\hbar_n\eps_{\hbar_n}^{-\nu_0})e^{-\frac{it\tau_{\hbar_n}}{\hbar_n}(\hat{P}_0(\hbar_n)+\eps_{\hbar_n}V)}\psi_{\hbar_n}\right\|^2dt.$$
In particular, one has
$$1\leq
\limsup_{n\rightarrow +\infty}\left\|\Pi(1,\hbar_n\eps_{\hbar_n}^{-\nu_0})e^{-\frac{ic|\log\eps_{\hbar_n}|}{\hbar_n}(\hat{P}_0(\hbar_n)+\eps_{\hbar_n}V)}\Pi(1,\hbar_n\eps_{\hbar_n}^{-\nu_0})\right\|^2,$$
which is $\leq c_0^2<1$ from lemma~\ref{l:loschmidt}, and thus provides the contradiction.
\end{proof}

\appendix

\section{Semiclassical analysis on manifolds}

\label{a:pdo}
In this appendix, we review some basic facts on semiclassical analysis that can be found for instance in~\cite{Zw12} -- chapter~$14$.

\subsection{General facts}

Recall that we define on $\mathbb{R}^{2d}$ the following class of admissible symbols:
$$S^{m,k}(\mathbb{R}^{2d}):=\left\{(a_{\hbar}(x,\xi))_{\hbar\in(0,1]}\in \ml{C}^{\infty}(\mathbb{R}^{2d}):|\partial^{\alpha}_x\partial^{\beta}_{\xi}a_{\hbar}|
\leq C_{\alpha,\beta}\hbar^{-k}\langle\xi\rangle^{m-|\beta|}\right\}.$$
Let $M$ be a smooth Riemannian $d$-manifold without boundary. Consider a smooth atlas $(f_l,V_l)$ of $M$, where each $f_l$ is a smooth diffeomorphism from 
$V_l\subset M$ to a bounded open set $W_l\subset\mathbb{R}^{d}$. To each $f_l$ correspond a pull back $f_l^*:\ml{C}^{\infty}(W_l)\rightarrow \ml{C}^{\infty}(V_l)$ and a canonical 
map $\tilde{f}_l$ from $T^*V_l$ to $T^*W_l$:
$$\tilde{f}_l:(x,\xi)\mapsto\left(f_l(x),(Df_l(x)^{-1})^T\xi\right).$$
Consider now a smooth locally finite partition of identity $(\phi_l)$ adapted to the previous atlas $(f_l,V_l)$. 
That means $\sum_l\phi_l=1$ and $\phi_l\in \ml{C}^{\infty}(V_l)$. Then, any observable $a$ in $\ml{C}^{\infty}(T^*M)$ can be decomposed as follows: $a=\sum_l a_l$, where 
$a_l=a\phi_l$. Each $a_l$ belongs to $\ml{C}^{\infty}(T^*V_l)$ and can be pushed to a function $\tilde{a}_l=(\tilde{f}_l^{-1})^*a_l\in \ml{C}^{\infty}(T^*W_l)$. 
As in~\cite{Zw12}, define the class of symbols of order $m$ and index $k$
\begin{equation}
\label{defpdo}S^{m,k}(T^{*}M):=\left\{(a_{\hbar}(x,\xi))_{\hbar\in(0,1]}\in \ml{C}^{\infty}(T^*M):|\partial^{\alpha}_x\partial^{\beta}_{\xi}a_{\hbar}|\leq C_{\alpha,\beta}\hbar^{-k}\langle\xi\rangle^{m-|\beta|}\right\}.
\end{equation}
Then, for $a\in S^{m,k}(T^{*}M)$ and for each $l$, one can associate to the symbol $\tilde{a}_l\in S^{m,k}(\mathbb{R}^{2d})$ the standard Weyl quantization
$$\Op_{\hbar}^{w}(\tilde{a}_l)u(x):=
\frac{1}{(2\pi\hbar)^d}\int_{\IR^{2d}}e^{\frac{\imath}{\hbar}\langle x-y,\xi\rangle}\tilde{a}_l\left(\frac{x+y}{2},\xi;\hbar\right)u(y)dyd\xi,$$
where $u\in\mathcal{S}(\mathbb{R}^d)$, the Schwartz class. Consider now a smooth cutoff $\psi_l\in \ml{C}_c^{\infty}(V_l)$ such that $\psi_l=1$ close to the support of $\phi_l$. 
A quantization of $a\in S^{m,k}(T^*M)$ is then defined in the following way~\cite{Zw12}:
\begin{equation}
\label{pdomanifold}\Op_{\hbar}(a)(u):=\sum_l \psi_l\times\left(f_l^*\Op_{\hbar}^w(\tilde{a}_l)(f_l^{-1})^*\right)\left(\psi_l\times u\right),
\end{equation}
where $u\in \ml{C}^{\infty}(M)$. This quantization procedure $\Op_{\hbar}$ sends (modulo $\mathcal{O}(\hbar^{\infty})$) $S^{m,k}(T^{*}M)$ onto the space of pseudodifferential 
operators of order $m$ and of index $k$, denoted $\Psi^{m,k}(M)$~\cite{Zw12}. It can be shown that the dependence in the cutoffs $\phi_l$ and $\psi_l$ only appears at order 
$1$ in $\hbar$ (Theorem $18.1.17$ in~\cite{Ho85} or Theorem $9.10$ in~\cite{Zw12}) and the principal symbol map $\sigma_0:\Psi^{m,k}(M)\rightarrow S^{m,k}/S^{m-1,k-1}(T^{*}M)$ is 
intrinsically defined. Most of the rules (for example the composition of operators, the Egorov and Calder\'on-Vaillancourt Theorems) that hold on 
$\mathbb{R}^{2d}$ still hold in the case of $\Psi^{m,k}(M)$. Because our study concerns the behavior of quantum evolution for logarithmic times in $\hbar$, a larger class of 
symbols should be introduced as in~\cite{Zw12}, for $0\leq\overline{\nu}<1/2$,
\begin{equation}\label{symbol}
S^{m,k}_{\overline{\nu}}(T^{*}M):=\left\{(a_{\hbar})_{\hbar\in(0,1]}\in C^{\infty}(T^*M):
|\partial^{\alpha}_x\partial^{\beta}_{\xi}a_{\hbar}|\leq C_{\alpha,\beta}\hbar^{-k-\overline{\nu}|\alpha+\beta|}\langle\xi\rangle^{m-|\beta|}\right\}.
\end{equation}
Results of~\cite{Zw12} (as Calder\'on-Vaillancourt and Egorov theorems) can also be applied to this new class of symbols. 

\subsection{Egorov theorem}\label{ss:egorov} 

In this paragraph, we briefly recall the Egorov theorem for short logarithmic 
times on constant negatively curved surfaces. The proof of this result on compact manifold was given in~\cite{AnNo07, DyGu14} 
building on earlier proofs on~$\IR^d$~\cite{BaGrPa99, BoRo02, Zw12}. Here, we will in fact be interested in generalizations of the 
results from~\cite{AnNo07, DyGu14} in the context of ``perturbed'' Schr\"odinger operators $\hat{P}_{\eps}(\hbar):=\hat{P}_0(\hbar)+\eps V$. It was 
explained in appendix $B$ of~\cite{EsRi14} how the Egorov Theorem for short logarithmic 
times can be extended to $\hat{P}_{\eps}(\hbar)$ by slightly adapting the arguments 
from~\cite{AnNo07, DyGu14}. We will just recall the result we need, and we refer the reader to the above references for more details.

Let $\delta>0$ and let $a$ be a smooth function which is compactly supported in the following neighborhood of $S^*M$ of size $\delta$, i.e.:
$$T_{[1/2-\delta,1/2+\delta]}^*M:=\left\{(x,\xi)\in T^*M: p_0(x,\xi)\in[1/2-\delta, 1/2+\delta] \right\},$$
where $p_0(x,\xi):=\frac{\|\xi\|^2}{2}$. According to~\cite{EsRi14} (appendix $B$), there exists $\eps_0>0$ such that, for every $\eps\in[0,\eps_0]$ and for 
every smooth function $a$ compactly supported in $T_{[1/2-\delta,1/2+\delta]}^*M$, one has that, for every 
$0\leq t\leq \frac{(1-\delta)|\log\hbar|}{2\sqrt{1+6\delta}}$, the operator
$$e^{\frac{it\hat{P}_{\eps}(\hbar)}{\hbar}}\Oph(a)e^{-\frac{it\hat{P}_{\eps}(\hbar)}{\hbar}}$$
belongs to $\Psi^{-\infty,0}_{\overline{\nu}}(M)$ for some $0<\overline{\nu}<1/2$ (which depends on $\delta$ and $\eps_0$). Moreover, its principal symbol is equal to $a\circ G_{\eps}^t$ and all 
the involved semi-norms can be uniformly bounded in terms of $\eps\in[0,\eps_0]$ and of $t$ in the above range. Also, thanks to the Calder\'on-Vaillancourt Theorem, 
we find that, uniformly for $\eps\in[0,\eps_0]$ and $0\leq t\leq \frac{(1-\delta)|\log\hbar|}{2\sqrt{1+6\delta}}$, one has
\begin{equation}\label{e:large-time-egorov}
\left\|e^{\frac{it\hat{P}_{\eps}(\hbar)}{\hbar}}\Oph(a)e^{-\frac{it\hat{P}_{\eps}(\hbar)}{\hbar}}-\Oph(a\circ G_{\eps}^t)\right\|_{L^2\rightarrow L^2}=o(1),
\end{equation}
as $\hbar\rightarrow 0$.

\subsection{Positive quantization}\label{ss:antiwick}

Even if the Weyl procedure is a natural choice to quantize an observable $a$ on $\mathbb{R}^{2d}$, it is sometimes preferrable to use a quantization 
procedure $\Op_{\hbar}^+$ that satisfies the following property~: $\Op_{\hbar}^+(a)\geq 0$ if $a\geq0$. This can be achieved thanks to the anti-Wick procedure $\Op_{\hbar}^{AW}$, 
see~\cite{HeMaRo87} for instance. For $a$ in $S^{0,0}_{\overline{\nu}}(\mathbb{R}^{2d})$, that coincides with a function on $\mathbb{R}^d$ outside a compact subset of 
$T^*\mathbb{R}^d=\mathbb{R}^{2d}$, one has
\begin{equation}\|\Op_{\hbar}^w(a)-\Op_{\hbar}^{AW}(a)\|_{L^2}\leq C\sum_{1\leq|\alpha|\leq D}\hbar^{\frac{|\alpha|}{2}}
\|\partial^{\alpha}a\|,
\end{equation}
where $C$ and $D$ are some positive constants that depend only on the dimension $d$.
To get a positive procedure of quantization on a manifold, one can replace the Weyl quantization by the anti-Wick one in definition~(\ref{pdomanifold}). 
This new choice of quantization (that we will denote by $\Oph^+$) is positive and it is well defined for every element $a$ in $S^{0,0}_{\overline{\nu}}(T^*M)$ of the 
form $c_0(x)+c(x,\xi)$ where $c_0$ belongs to 
$S^{0,0}_{\overline{\nu}}(T^*M)$ and $c$ belongs to\footnote{Here we mean that there exists a compact subset $K\subset T^*M$ such that $\text{supp}(a_{\hbar})\subset K$ for every $0<\hbar\leq 1$.} $\mathcal{C}^{\infty}_c(T^*M)\cap S^{0,0}_{\overline{\nu}}(T^*M)$. We can also require that $\Oph^+(1)=\text{Id}_{L^2(M)}.$ The main observation is that, for such symbols, one 
has
\begin{equation}\label{e:positive-quantization}
\left\|\Oph^+(a)-\Oph(a)\right\|_{L^2}\leq C'\sum_{1\leq|\alpha|\leq D'}\hbar^{\frac{|\alpha|}{2}}\|\partial^{\alpha}a\|,
\end{equation}
where $C'$ and $D'$ are some positive constants that depend only on the manifold $M$ and on the choice of coordinate charts.

\section{Strong structural stability}
\label{a:stability}

In our proof, we needed to use the strong structural stability property for Anosov flows~\cite{Ano67}, and more precisely, we needed to use the fact that the conjugating homeomorphism and the reparametrization function depend in a ``smooth'' way on the perturbation parameter $\eps$. This regularity was observed by De La Llave, Marco and Moriyon~\cite{dLLMM86} based on ``analytic'' proofs of structural stability for Anosov diffeomorphisms due to Moser~\cite{Mo69} and Mather~\cite{Sm67}. In this appendix, we recall a few facts on the geometric properties of the conjugating homeomorphism that were proved in section $5$ of~\cite{EsRi14} based on the arguments of~\cite{dLLMM86}.

First, we briefly recall strong structural stability property for Anosov flows in the same way as it was stated in~\cite{dLLMM86}. For that purpose, we introduce some manifolds of mappings that will be involved in this theorem -- see~\cite{Ee58, Ab63, Ee66} or the appendix of~\cite{EsRi14} for a brief reminder on their differential structure. Define
$$\mathcal{C}_{X_0}(S^*M):=\left\{h\in\ml{C}^{0}(S^*M,S^*M):\ \forall \rho\in S^*M,
\ \left(\frac{d}{dt}h\circ G_0^t(\rho)\right)_{t=0}=D_{X_0}h\ \text{exists}\right\},$$
which can be endowed with a smooth differential structure modeled on the Banach spaces of continuous sections $s:S^*M\mapsto h^*TS^*M$ which are differentiable along the geodesic flow. This manifold contains an ``adapted'' submanifold $\ml{M}$ which contains $\text{Id}_{S^*M}$ in its interior and \emph{for which the elements are in some sense ``transversal'' to the geodesic vector field} $X_0$~\cite{dLLMM86} -- appendix $A$ (see also appendix of~\cite{EsRi14}). 

The structural stability theorem can be then stated as follows (theorem $A.2$ in~\cite{dLLMM86}):

\begin{theo}\label{t:stability}[Strong structural stability] Assume $X_0$ is an Anosov vector field. There exists an open neighborhood $\ml{U}_0(X_0)$ of $X_0$ in $\ml{V}^2(S^*M)$ and a 
unique $\ml{C}^2$ map $S_0:\ml{U}_0(X_0)\rightarrow \ml{M}\times\ml{C}^0(S^*M,\IR)$ such that $S_0(X_0)=(\text{Id}_{S^*M},1)$ and if $S_0(X)=(h,\tau)$, then
\begin{equation}
 \label{e:implicit-equation}
D_{X_0}h-\tau X\circ h=0_{S^*M}(h),
\end{equation}
where $0_{S^*M}$ is the zero section.
\end{theo}

\begin{rema}\label{r:reparametrization} In fact, the neighborhood $\ml{U}_0(X_0)$ can be chosen small enough to ensure that $h$ is an homeomorphism -- appendix $A$ of~\cite{dLLMM86} or remark $5.5$ of~\cite{EsRi14}. Using that $h$ is an homeomorphism, we can also write the following formula connecting the two flows:
\begin{equation}\label{e:struc-stab-useful} 
\forall\ t\in\IR,\ h\circ G_0^{\tau(t,\rho)}\circ h^{-1}(\rho)=G_X^t(\rho),
\end{equation}
where $$\tau(t,\rho):=\int_{0}^t\frac{ds}{\tau\circ h^{-1}\circ G_X^s(\rho)}.$$
\end{rema}

We now describe some geometric properties of the conjugating homeomorphism when we apply the strong structural stability theorem to the perturbations $Y_{x_1,\xi_1}^{\eps}$. Recall that they define small $\ml{C}^1$ perturbation of the geodesic vector field $X_0$. We refer to section $5$ of~\cite{EsRi14} for the details of the proofs. 

Observe that there exists $\eps_0>0$ such that, for every $(x_1,\xi_1)$ in a small neighborhood of $S^*M$ and for every $\eps\in[0,\eps_0]$, $Y_{x_1,\xi_1}^{\eps}$ belongs to the neighborhood of the previous theorem. We write
$$S_0(Y_{x_1,\xi_1}^{\eps})=(h_{x_1,\xi_1}^{\eps},\tau_{x_1,\xi_1}^{\eps}).$$
As the map $S_0$ is of class $\ml{C}^1$, we can write that
\begin{equation}\label{e:smoothmap}
\sup_{\rho\in S^*M}\left\{d(h_{x_1,\xi_1}^{\eps}(\rho),\rho),|\tau_{x_1,\xi_1}^{\eps}(\rho)-1|\right\}=\ml{O}(\eps),
\end{equation}
where the constant in the remainder is uniform for $(x_1,\xi_1)$ in a small neighborhood of $S^*M$. In our proof, we also need to understand precisely the properties of the map $(h_{x_1,\xi_1}^{\eps})^{-1}$. We observe that the map $\eps\mapsto (h_{x_1,\xi_1}^{\eps})^{-1}\in\ml{C}^0(S^*M,S^*M)$ has a priori no reason to be of class $\ml{C}^1$ -- see remark $5.7$ in~\cite{EsRi14}. In order to solve this problem, we will write
$$ h_{x_1,\xi_1}^{\eps}:=\exp(v_{x_1,\xi_1}^{\eps}),$$
where $v_{x_1,\xi_1}^{\eps}$ is a continuous vector field and $\exp$ is the exponential map induced by the Riemannian structure on $S^*M$. We also introduce the following vector field on $S^*M$:
\begin{equation}\label{e:approx-inverse}\tilde{v}_{x_1,\xi_1}^{\eps}:=\frac{\eps}{\|\xi_1\|}\left(\beta^s_VX^s+\beta^u_VX^u\right),
\end{equation}
where
$$\beta^s_V(x,\xi):=\frac{1}{\sqrt{2}}\int_0^{+\infty}g_{x(-t)}^*\left(d_{x(-t)}V,\xi^{\perp}(-t)\right)e^{-t}dt,$$
and
$$\beta^u_V(x,\xi):=\frac{1}{\sqrt{2}}\int_0^{+\infty}g_{x(t)}^*\left(d_{x(t)}V,\xi^{\perp}(t)\right)e^{-t}dt,$$
with $G_0^t(x,\xi):=(x(t),\xi(t)).$ We observe that these two functions do not depend on $(x_1,\xi_1)$. It was proved in~\cite{EsRi14} that $\beta_{V}^u$ and 
$\beta_{V}^s$ are $\ml{C}^{\gamma}$-H\"older for every $\gamma<1/2$ -- lemma $5.13$ from this reference. According to lemma~$1$ in~\cite{Mo69}, we can write that 
\begin{equation}\label{e:prod-homeo}\exp (-\tilde{v}_{x_1,\xi_1}^{\eps})\circ\exp v_{x_1,\xi_1}^{\eps}=\exp(-\tilde{v}_{x_1,\xi_1}^{\eps}+v_{x_1,\xi_1}^{\eps}+r(\tilde{v}_{x_1,\xi_1}^{\eps},v_{x_1,\xi_1}^{\eps})),\end{equation} 
where $$\|r(\tilde{v}_{x_1,\xi_1}^{\eps},v_{x_1,\xi_1}^{\eps})\|_{\ml{C}^0}\leq C\|\tilde{v}_{x_1,\xi_1}^{\eps}\|_{\ml{C}^{\gamma}}\|v_{x_1,\xi_1}^{\eps}\|_{\ml{C}^0}^{\gamma},$$
for some uniform constant $C>0$ (depending on the manifold and on $\gamma$). In particular, we have that, in our setting, $\|r(\tilde{v}_{x_1,\xi_1}^{\eps},v_{x_1,\xi_1}^{\eps})\|_{\ml{C}^0}=\ml{O}(\eps^{1+\gamma})$ with the 
constant involved in the remainder which is uniform for $(x_1,\xi_1)$ in a small neighborhood of $S^*M$.

\begin{rema}\label{r:moser} The proof of this fact was given in the appendix of~\cite{Mo69} for the general case of vector fields on a Riemannian manifold. The only difference is that 
the proof given in this reference is for $\gamma=1$. Yet, the proof can be directly adapted to get the above estimate involving H\"older norms. 
\end{rema}

Finally, according to paragraphs $5.3.1$ and $5.3.2$ in~\cite{EsRi14}, one has that $\tilde{v}_{x_1,\xi_1}^{\eps}$ is equal to $v_{x_1,\xi_1}^{\eps}$ up to an error of order $\ml{O}(\eps^2)$ in the $\ml{C}^0$-topology. This property followed from the differentiation of the implicit equation~\eqref{e:implicit-equation}. In particular, thanks to~\eqref{e:prod-homeo}, we can derive that
$\exp (-\tilde{v}_{x_1,\xi_1}^{\eps})\circ h_{x_1,\xi_1}^{\eps}$ is close to identity up to an error of order $\ml{O}(\eps^{1+\gamma})$, where the constant involved is uniform for $(x_1,\xi_1)$ in a small neighborhood of $S^*M$. We underline that this property holds for every $0<\gamma<1/2$. To summarize, for every $0<\gamma<1/2$, we have that
\begin{equation}\label{e:inverse-holder}
\sup_{\rho\in S^*M}\left\{d\left(\left(h_{x_1,\xi_1}^{\eps}\right)^{-1}(\rho),\exp (-\tilde{v}_{x_1,\xi_1}^{\eps})(\rho)\right)\right\}=\ml{O}(\eps^{1+\gamma}),
\end{equation}
where the constant in the remainder is uniform for $(x_1,\xi_1)$ in a small neighborhood of $S^*M$. Thus, even if the map $\eps\mapsto (h_{x_1,\xi_1}^{\eps})^{-1}\in\ml{C}^0(S^*M,S^*M)$ is not ``smooth'', it can be approximated in a precise way by a smooth map which has a very explicit expression.

\section{Potentials satisfying~$\ml{C}_V=\emptyset$}
\label{a:example}
In this appendix, we will prove the following proposition which shows that the assumptions appearing in section~\ref{s:loschmidt} and corollary~\ref{c:coro1} are in some sense ``generic'':
\begin{prop}\label{p:geom-prop} Let $M$ be a smooth compact oriented Riemannian boundaryless surface. Then, the set
$$\ml{U}:=\left\{V\in\mathcal{C}^{\infty}(M,\IR):\ml{C}_V=\emptyset\right\}$$
 is open and dense in $\mathcal{C}^{\infty}(M,\IR)$ endowed with its natural topology of Fr\'echet space.
\end{prop}
We note that we do not require $M$ to be negatively curved in this statement. Recall that we have defined $f_V(x,\xi):=g_x^*(d_xV,\xi^{\perp})$, and 
$$\ml{C}_V:=\left\{\rho\in S^*M:\ \forall\ j\geq 0,\ X_0^j.f_V(\rho)= 0\right\}.$$ 
The proof below was indicated to us by Jean-Yves Welschinger. 

\begin{proof}
 Recall that the (natural) Fr\'echet topology on $\mathcal{C}^{\infty}(M,\IR)$ is in fact equivalent to the topology induced by the following metric:
$$\forall V,W\in\mathcal{C}^{\infty}(M,\IR),\ D(V,W):=\sum_{j\geq 0}\frac{1}{2^{j+1}}\min\left\{1,\|V-W\|_{\ml{C}^j}\right\},$$
where, for every $i\geq 0$, $\|.\|_{\ml{C}^i}$ is the usual norm on $\ml{C}^i(M,\IR)$.

We will first prove that the set $\ml{U}$ is open. We fix $V$ in $\ml{U}$. By compactness, we observe that there exists some $J_0>0$ such that, for every $\rho$ in $S^*M$, there exists $0\leq j\leq J_0$ such that $X_0^j.f_V(\rho)\neq 0$. We introduce
$$\tilde{f}_V(\rho):=\max\left\{|X_0^j.f_V(\rho)|:0\leq j\leq J_0\right\},$$
which is continuous on $S^*M$. From our assumption, there exists $0<\delta_0<1$ such that $\tilde{f}_V\geq \delta_0$ on $S^*M$. We now observe that, for every $j\geq 0$, there exists a constant $c_j\geq 1$ (depending only on $(M,g)$ and on $j$) such that, for every 
$\rho \in S^*M$ and for every $W$ in $\mathcal{C}^{\infty}(M,\IR)$, one has
$$|X_0^j.(f_V-f_W)(\rho)|\leq c_j\|V-W\|_{\ml{C}^{j+1}}.$$
If we take $W$ in $\mathcal{C}^{\infty}(M,\IR)$ such that $D(V,W)\leq \frac{\delta_0}{2^{J_0+2}\max_{0\leq j\leq J_0}c_j}$, then, for every $0\leq j\leq J_0$, one has $\|V-W\|_{\ml{C}^j}\leq\frac{\delta_0}{2\max_{0\leq j\leq J_0}c_j}$. Then, we deduce that $W$ belongs to $\ml{U}$.

It remains to show that the set $\ml{U}$ is dense. This will follow from the Sard-Smale's theorem~\cite{Sm65}. Before applying this theorem, we make a simple observation. We fix $V$ in $\ml{C}^{\infty}(M,\IR)$ and $\delta_0>0$. We observe that there exists $J_0\geq 5$ such that, for every $\tilde{W}$ in $\mathcal{C}^{\infty}(M,\IR)$,
$$D(V,\tilde{W})\leq \sum_{j=0}^{J_0}\frac{1}{2^{j+1}}\|V-\tilde{W}\|_{\ml{C}^j}+ \delta_0.$$
Suppose now that we are able to find $W$ in $\ml{C}^{J_0}(M,\IR)$ which is $\delta_0$ close to $V$ in the $\ml{C}^{J_0}$ topology and such that, for every $\rho\in S^*M$, there exists $0\leq j\leq 3$ verifying $X_0^j.f_W(\rho)\neq 0$. Then, we can use the fact that $\ml{C}^{\infty}(M,\IR)$ is dense in $\ml{C}^{J_0}(M,\IR)$ for the $\ml{C}^{J_0}$ topology and conclude. In fact, by density, we can find, for every $0<\delta<\delta_0$, $\tilde{W}$ in $\ml{C}^{\infty}(M,\IR)$ such that $\|\tilde{W}-W\|_{\ml{C}^{J_0}}\leq \frac{\delta}{\max_{0\leq j\leq J_0}c_j}$. In particular, for such a $\tilde{W}$, one has $D(V,\tilde{W})\leq 3\delta_0$, and 
$$\forall\rho\in S^*M,\ \max_{0\leq j\leq 3}\{|X_0^j.f_{\tilde{W}}(\rho)|\}\geq \max_{0\leq j\leq 3}\{|X_0^j.f_{W}(\rho)|\}-\delta.$$
Taking $\delta>0$ small enough to ensure that the above quantity is positive for every $\rho$ in $S^*M$, we have found $\tilde{W}$ in $\ml{C}^{\infty}(M,\IR)$ which is $\delta_0$ close to $V$ in the $\ml{C}^{\infty}$ topology and such that, for every $\rho$ in $S^*M$, there exists $0\leq j\leq 3$ satisfying $X_0^j.f_W(\rho)\neq 0$.

It now remains to prove that, for every $J_0\geq 5$, the set
$$\ml{U}_{J_0}:=\left\{V\in\ml{C}^{J_0}(M,\IR):\ \forall\rho\in S^*M,\ \exists 0\leq j\leq 3\ \text{s.t.}\ X_0^j.f_V(\rho)\neq 0\right\}$$
contains a dense subset of $\ml{C}^{J_0}(M,\IR)$. For that purpose, we define
$$\ml{L}:(V,\rho)\in\ml{C}^{J_0}(M,\IR)\times S^*M\mapsto (f_V(\rho),X_0.f_V(\rho),X_0^2.f_V(\rho),X_0^3.f_V(\rho))\in\IR^4.$$
This defines a $\ml{C}^1$ map on the Banach manifold $\ml{C}^{J_0}(M,\IR)\times S^*M$, and we say that $0$ is a regular value of $\ml{L}$ if, for every $(V,\rho)$ satisfying $\ml{L}(V,\rho)=0$, the tangent map $D_{(V,\rho)}\ml{L}$ is a continuous surjective linear map whose kernel has a closed complement. One can verify that it is in fact continuous, and that the kernel has a closed complement (the kernel has finite codimension). In order to verify the surjectivity, it is sufficient to show that
\begin{equation}\label{e:surjectivity}\forall z\in \IR^4,\ \exists W\in \ml{C}^{J_0}(M,\IR)\ \text{such that}\ D_{(V,\rho)}\ml{L}.(W,0)=z.\end{equation}
We note that $D_{(V,\rho)}\ml{L}.(W,0)=\ml{L}(W,\rho)$, and we introduce the following linear maps, for every $\rho_0=(x_0,\xi_0)$ in $S^*M$:
$$L_{\rho_0}^{(0)}:W\in\ml{B}_0:=\ml{C}^{J_0}(M,\IR)\mapsto f_W(\rho_0)\in\IR,$$
and, for $1\leq j\leq 3$,
$$L_{\rho_0}^{(j)}:W\in \ml{B}_j:=\cap_{l=0}^{j-1}\text{Ker} \left(L_{\rho_0}^{(l)}\right) \mapsto X_0^j.f_W(\rho_0)\in\IR.$$
For every $0\leq j\leq 3$, one can verify that these maps are \emph{nonvanishing} linear forms on the infinite dimensional Banach space $\ml{B}_j$ -- see remark~\ref{r:nonvanish} below. In particular, taking $W_j$ which does not belong to the kernel of $L_{\rho_0}^{(j)}$ for every $0\leq j\leq 3$, we can write that, for every $0\leq j\leq 3$,
$$\ml{L}(W_j,\rho_0)=(a_0^j, a_1^j, a_2^j, a_3^j),$$
where, by construction, $a_j^j\neq 0$, and $a_p^j=0$ for every $ p<j$. In particular, $\IR^4=\text{span}\{D_{(V,\rho)}\ml{L}.(W_j,0):0\leq j\leq 3\}$, and the map $D_{(V,\rho)}\ml{L}$ is surjective.
\begin{rema}\label{r:nonvanish} Above, we claimed that $L_{\rho_0}^{(j)}$ is a \emph{nonvanishing} linear form when it acts on the Banach space $\ml{B}_j\subset\ml{C}^{J_0}(M,\IR)$. This can be proved as follows. Let $0\leq j\leq 3$. Consider $\kappa: U\rightarrow V\subset\IR^2$ a chart centered at $x_0$ in $M$. This chart can be lifted to a chart centered at $(x_0,0)$ in $T^*M$ as follows
$$\tilde{\kappa}:T^*U\rightarrow T^*V\subset\IR^{4},\ \ \ (x,\xi)\mapsto (u^k,v_l):=(\kappa(x),(d\kappa(x)^T)^{-1}\xi).$$
For a fixed smooth function $W$ on $M$, we define in a small neighborhood of $0$ the function $\tilde{W}(u^k)=W\circ\kappa^{-1}(u^k).$ In these local coordinates, the map $f_W$ can rewritten:
$$f_W(x,\xi):=\sum_{k}\left(\sum_{l}\tilde{g}^{k,l}(u)v_l\right)\partial_k\tilde{W}(u).$$ 
where $(\tilde{g}^{k,l})_{k,l}$ is a non-degenerate $2$-form. For $\xi$ not equal to $0$, we note that at least one of the two coefficients $\tilde{v}_k:=\sum_{l}\tilde{g}^{k,l}(u)v_l$ does not vanish. In particular, at the point $(x_0,\xi_0)$, one has
$$f_W(x_0,\xi_0):=d_0\tilde{W}(\tilde{v}),$$ 
with $\tilde{v}\neq 0$ independent of $W$ (as $\xi_0\neq 0$). More generally, for $1\leq j\leq 3$, one can write in local coordinates,
$$X_0^j.f_W(x_0,\xi_0):=d^{j+1}\tilde{W}(\tilde{v},v,\ldots,v)+\sum_{\alpha\in\mathbb{N}^2:|\alpha|< j+1}a_{\alpha}^j\partial^{\alpha}\tilde{W}(0,0),$$
where $a_{\alpha}^j$ are real numbers which are independent of $W$. If we choose $W$ such that $d_{x_0}^pW=0$ for every $p<j$, then $W$ belongs to the space $\ml{B}_j$, and we have, for such a function $W$,
$$L_{\rho_0}^{(j)}(W):=X_0^j.f_W(x_0,\xi_0):=d^{j+1}\tilde{W}(\tilde{v},v,\ldots,v).$$
As the vectors $v$ and $\tilde{v}$ are both nonzero, we can find a function $\tilde{W}$ whose derivatives up to order $j$ vanish at $(0,0)$ and such that the previous quantity does not vanish.
\end{rema}

From the previous discussion, we can conclude that $\ml{L}^{-1}(0)$ defines a $\ml{C}^1$ submanifold of codimension $4$~\cite{La02} -- Ch.~$1$ and~$2$. We are now in position to conclude. For that purpose, we define the projection map
$$\Lambda:(V,\rho)\in\ml{C}^{J_0}(M,\IR)\times S^*M\mapsto V\in\ml{C}^{J_0}(M,\IR).$$
In the terminology\footnote{Recall that a Fredholm map is a $\ml{C}^1$ map whose tangent map defines a Fredholm operator. The index of a Fredholm map is then the index of its tangent map.} of~\cite{Sm65}, this defines a Fredholm map of index $3$. One can also consider the restriction $\Lambda\rceil_{\ml{L}^{-1}(0)}$ of this map to the codimension $4$ submanifold $\ml{L}^{-1}(0)$. This map can also be written $\Lambda\rceil_{\ml{L}^{-1}(0)}=\Lambda\circ I$, where $I$ is the inclusion map from $\ml{L}^{-1}(0)$ to $\ml{C}^{J_0}(M,\IR)\times S^*M$. Thanks to the above observation, $I$ is a Fredholm map of index $-4$. By the composition rules for Fredholm operators, one has that $\Lambda\rceil_{\ml{L}^{-1}(0)}$ is a Fredholm map of index $-1$. Thanks to the Sard-Smale's theorem -- for instance corollary $1.5$ in~\cite{Sm65}, there exists a dense subset $\ml{D}$ of $\ml{C}^{J_0}(M,\IR)$ such that, for every $V$ in $\ml{D}$, $(\Lambda\rceil_{\ml{L}^{-1}(0)})^{-1}(V)$ is empty. In particular, every $V$ in $\ml{D}$ belongs to $\ml{U}_{J_0}$; thus, $\ml{U}_{J_0}$ is a dense subset.

\end{proof}


\section*{Acknowledgements}
The author is partially supported by the Agence Nationale de la Recherche through the Labex CEMPI (ANR-11-LABX-0007-01) and the ANR project GeRaSic (ANR-13-BS01-0007-01). We warmly thank Suresh Eswarathasan for his comments on a preliminary version of this work, and Jean-Yves Welschinger for indicating us the proof of proposition~\ref{p:geom-prop}.

\end{document}